\documentclass[12pt,a4paper]{article}%
\usepackage{amsmath}
\usepackage{amssymb}
\usepackage{makeidx}
\usepackage{amsfonts}
\usepackage{hyperref}
\usepackage{graphicx}%
\providecommand{\U}[1]{\protect\rule{.1in}{.1in}}
\newtheorem{theorem}{Theorem}[section]

\newtheorem{example}[theorem]{Example}
\newtheorem{lemma}[theorem]{Lemma}
\newtheorem{proposition}[theorem]{Proposition}

\newenvironment{proof}[1][Proof]{\noindent\textbf{#1.} }{\rule{0.5em}{0.5em}}

\begin{document}

\title{On the spectrum of the hierarchical Schr\"{o}dinger operator }
\author{Alexander Bendikov\thanks{A. Bendikov was supported by the Polish National
Science Center, Grant 2015/17/B/ST 1/00062}
\and Alexander Grigor'yan\thanks{A. Grigor'yan was supported by SFB 1283 of the
German Research Council.}
\and Stanislav Molchanov}
\maketitle

\begin{abstract}
The goal of this paper is the spectral analysis of the Schr\"{o}dinger
operator $H=L+V$ , the perturbation of the Taibleson-Vladimirov multiplier
$L=\mathcal{D}^{\alpha}$ by a potential $V$. Assuming that $V$ belonges to a
class of fast decreasing potentials we show that the discrete part of the
spectrum of $H$ may contain negative energies, it also appears in the spectral
gaps of $L$. We will split the spectrum of $H$ in two parts: high energy part
containing eigenvalues which correspond to the eigenfunctions located on the
support of the potential $V,$ and low energy part which lies in the spectrum
of certain bounded Schr\"{o}dinger operator acting on the Dyson hierarchical
lattice. The spectral asymptotics \ strictly depend on the transience versus
recurrence properties of the underlying hierarchical random walk. In the
transient case we will prove results in spirit of CLR theory, for the
recurrent case we will provide Bargmann's type asymptotics.

\end{abstract}
\tableofcontents

\section{Introduction}

\setcounter{equation}{0}

The spectral theory of nested fractals similar to the Sierpinski gasket, i.e.
the spectral theory of the corresponding Laplacians, is well understood. It
has several important features: Cantor-like structure of the essential
spectrum and, as result, the large number of spectral gaps, presence of
infinite number of eigenvalues each of which has infinite multiplicity and
compactly supported eigenstates, non-regularly varying at infinity heat
kernels which contain oscilated in $\log t$ scale terms etc, see
\cite{GrabnerWoess}, \cite{DerfelGrabner} and \cite{BCW}.

The spectral properties mentioned above occure in the very precise form for
the Taibleson-Vladimirov Laplacian $\mathcal{D}^{\alpha}$, the operator of
fractional derivative of order $\alpha$. This operator can be introduced in
several different forms (say, as $L^{2}$-multiplier in the $p$-adic analysis
setting, see \cite{Vladimirov}) but we select the geometric approach
\cite{Dyson1}, \cite{Molchanov}, \cite{Molchanov1}, \cite{BGP}, \cite{BGPW},
\cite{BendikovKrupski} and \cite{BGMS}.

\paragraph{The Dyson's hierarchical model}

Let us fix an integer $p\geq2$ and consider the family $\{\Pi_{r}%
:r\in\mathbb{Z}\}$ of partitions of $X=\mathbb{[}0,+\infty\lbrack$ such that
each $\Pi_{r}$ consists of all $p$-adic intervals $[kp^{r},(k+1)p^{r}[$. We
call $r$ the rank of the partition $\Pi_{r}$\ (respectively, the rank of the
interval $I\in\Pi_{r}$). Each interval of rank $r$ is the union of $p$
disjoint intervals of rank $(r-1)$. Each point $x\in X$ \ belongs to a certain
interval $I_{r}(x)$ of rank $r$, and intersection of all $p$-adic intervals
$I_{r}(x)$ is $\{x\}.$

\emph{The hierarchical distance} $d(x,y)$ is defined as follows:%
\[
d(x,y)=p^{\mathfrak{n}(x,y)}\text{, \ where }\mathfrak{n}(x,y)=\inf\{r:y\in
I_{r}(x)\}.
\]
Since any two points $x$ and $y$ belong to a certain $p$-adic interval,
$d(x,y)<\infty$. Clearly $d(x,y)=0$ if and only if $x=y$, $d(x,y)=d(y,x)$,
and
\[
d(x,y)\leq\max\{d(x,z),d(z,y)\}
\]
\ for arbitrary $x,y$ and $z$ in $X$, i.e. $d(x,y)$ is an \emph{ultrametric}.

The set $X$ equipped with the ultrametric $d(x,y)$ is \emph{complete,
separable} and \emph{proper} metric space. In the metric space $(X,d)$ the set
of all open balls is countable, it coincides with the set of all $p$-adic
intervals. The Borel $\sigma$-algebra $\mathcal{B}(X,d)$ coincides with the
Borel $\sigma$-algebra $\mathcal{B}(X,\mathrm{d})$ corresponding to the
Eucledian distance $\mathrm{d}$.

Following \cite{MolchanovVainberg1}\ we introduce \emph{the hierarchical
Laplacian} $L$ defined pointwise as convexe linear combination of "elementary
Laplacians"%
\begin{equation}
(Lf)(x)=%
{\displaystyle\sum\limits_{r=-\infty}^{+\infty}}
C(r)\left(  f(x)-\frac{1}{m(I_{r}(x))}%
{\displaystyle\int\limits_{I_{r}(x)}}
fdm\right)  ,\text{ \ } \label{HL}%
\end{equation}
where $C(r)=(1-\kappa)\kappa^{r-1}$, and $m$ is the Lebesgues measure. Here
$\kappa\in(0,1)$ is the second parameter of \ the model. Recall that the first
parameter of the model is $p$, the integer which defines the family of
partitions $\{\Pi_{r}\}$. The series in (\ref{HL}) diverges in general but it
is finite and belongs to $L^{2}=L^{2}(X,m)$ for all $f$\ which take constant
values on the $p$-adic intervals of the rank $r$.

The operator $L$ admits a complete system of compactly supported
eigenfunctions. Indeed, let $I$ be a $p$-adic interval of rank $r$, let
$I_{1},I_{2},...,I_{p}$ be its $p$-adic subintervals of rank $r-1$. Let us
consider $p$ functions%
\[
\psi_{I_{i}}=\frac{1_{I_{i}}}{m(I_{i})}-\frac{1_{I}}{m(I)}.
\]
We have $L\psi_{I_{i}}=\kappa^{r-1}\psi_{I_{i}}$. Since $\sum_{i=1}^{p}%
\psi_{I_{i}}=0$ the rank of the system $\{\psi_{I_{i}}:i=1,2,...,p\}$ is
$p-1$. When $I$ runs over the set of all $p$-adic intervals the system of
eigenfunctions $\{\psi_{I_{i}}\}$ is complete in $L^{2}$ whence $L$ is
essentially self-adjoint operator having pure point spectrum
\[
Spec(L)=\{0\}\cup\{\kappa^{r}:r\in\mathbb{Z}\}.
\]
Each eigenvalue $\lambda_{r}=\kappa^{r-1}$ has infinite multiplicity. We will
see below that writing $\kappa=p^{-\alpha}$, i.e. setting $\alpha=\ln\frac
{1}{\kappa}/\ln p$, the operator $L$ coinsides with the Taibleson-Vladimirov
operator $\mathcal{D}^{\alpha}$, the operator of fractional derivative of
order $\alpha$. The constant $s_{h}=2\ln p/\ln\frac{1}{\kappa}$ is
called\emph{ spectral dimension} of the triple $(X,d,L)$. It gives the
on-diagonal asymptotics of the transition density $p(t,x,x)\asymp t^{-s_{h}%
/2}$ of the Markov semigroup $(e^{-tL})_{t>0}$, see \cite[Proposition
2.3]{MolchanovVainberg1} and \cite{BGPW}.

There are already several publications on the spectrum of the hierarchical
Laplacian acting on a general ultrametric measure space $(X,d,m)$
\cite{AlbeverioKarwowski}, \cite{AisenmanMolchanov}, \cite{Molchanov},
\cite{Molchanov1}, \cite{BGP}, \cite{BGPW}, \cite{BendikovKrupski},
\cite{BGMS}. Accordingly, the hierarchical Schr\"{o}dinger operator was
studied in \cite{Dyson2}, \cite{Molchanov}, \cite{MolchanovVainberg1},
\cite{MolchanovVainberg2}, \cite{Bovier}, \cite{Kvitchevski1},
\cite{Kvitchevski2}, \cite{Kvitchevski3} (the hierarchical lattice of Dyson)
and in \cite{Vladimirov94}, \cite{VladimirovVolovich}, \cite{Kochubey2004}
(the field of $p$-adic numbers).

By the general theory developed in \cite{BGP}, \cite{BGPW} and
\cite{BendikovKrupski}, any hierarchical Laplacian $L$ acts in $L^{2}(X,m),$
is essentially self-adjoint operator and can be represented in the form
\begin{equation}
Lf(x)=%
{\displaystyle\int\limits_{X}}
(f(x)-f(y))J(x,y)dm(y)\text{. \ } \label{Spectrum}%
\end{equation}
It has a pure point spectrum, and its Markov semigroup $(e^{-tL})_{t>0}$
admits with respect to $m$ a continuous transition density $p(t,x,y)$. It
turns out that in terms of certain (intrinsically related to $L$) ultrametric
$d_{\ast}$,%
\begin{equation}
\text{\ }J(x,y)=\int\limits_{0}^{1/\emph{d}_{\ast}(x,y)}N(x,\tau)d\tau\text{,}
\label{Jump-kernel}%
\end{equation}%
\begin{equation}
\emph{p}(t,x,y)=t\int\limits_{0}^{1/\emph{d}_{\ast}(x,y)}N(x,\tau)\exp
(-t\tau)d\tau, \label{d*-jump kernel}%
\end{equation}
and%
\[
\emph{p}(t,x,x)=\int\limits_{0}^{\infty}\exp(-t\tau)dN(x,\tau)
\]
where $N(x,\tau)$ is the so called\emph{ spectral function }related to $L$.
The analytic properties of the function $p(t,x,y)$ play essential role in the
study of the Schr\"{o}dinger operator $H=L+V$, see paper
\cite{MolchanovVainberg2}.

\paragraph{Notation}

For two positive functions $f$ and $g$ we write $f\asymp g$ if the ratio $f/g$
is bounded from above and from below by positive constants for a specified
range of variables. We write $f\sim g$ if the ratio $f/g$ tends to identity.

A non-decreasing function $N:\mathbb{R}_{+}\rightarrow\mathbb{R}_{+}$ is
called \emph{doubling }if the inequality
\[
N(2r)\leq CN(r)
\]
holds for all $r>0$ and some $C>1$. The doubling property implies that
\[
\frac{N(R)}{N(r)}\leq C^{\prime}\left(  \frac{R}{r}\right)  ^{\tau}%
\]
for all $R>r>0$ and some constants $\tau,C^{\prime}>0.$

A non-decreasing function $M:\mathbb{R}_{+}\rightarrow\mathbb{R}_{+}$ is
called \emph{reverse doubling }if the inequality%
\[
\frac{M(R)}{M(r)}\geq C^{\prime\prime}\left(  \frac{R}{r}\right)  ^{\nu},
\]
holds for all $R>r>0$ and some constants $C^{\prime\prime},\nu>0.$

\paragraph{Outline}

Let us describe the main body of the paper. In Section 2 we introduce the
notion of homogeneous hierarchical Laplacian $L$ and list its basic properties
such as: the set $Spec(L),$ the spectrum of the operator\emph{ }$L$, is pure
point, all eigenvalues of $L$ have infinite multiplicity and compactly
supported eigenfunctions, the heat kernel $p(t,x,y)$ exists and is a
continuous function having nice asymptotic properties etc.). As a special
example we consider the case $X=\mathbb{Q}_{p},$ the ring $\mathbb{Q}_{p}$ of
$p$-adic numbers, endowed with its standard ultrametric $d(x,y)=\left\vert
x-y\right\vert _{p}$ and the normed Haar measure $m$. The hierarchical
Laplacian $L$ in our example coincides with the Taibleson-Vladimirov operator
$\mathfrak{D}^{\alpha}$, the operator of fractional derivative of order
$\alpha$, see \cite{Vladimirov}, \cite{Vladimirov94}, and \cite{Kochubey2004}.
The most complete sourse for the basic definitions and facts related to the
$p$-adic analysis is \cite{Koblitz} and \cite{Taibleson75}.

In the next sections we consider the Schr\"{o}dinger operator $H=L+V$ with a
continuous descending at infinity potential of the form $V=%
{\displaystyle\sum}
\sigma_{i}1_{B_{i}}$, where $B_{i}$ are balls. The main aim here is to study
the set $Spec(H).$ Since $V(x)\rightarrow0$ as $x\rightarrow\infty$ the set
$Spec(H)$\ is pure point with possibly countaly many limit points - the
eigenvalues of the operator $L$. We splitt the set $Spec(H)$ in two disjoint
parts: the first part is related to $Spec(L)$ and the second part is countably
infinite set $\Xi.$ In the case when $d(B_{i},B_{j}),i\neq j,$ become large
enough we specify the structure of the set $\Xi$. We obtain in certain cases
lower bounds on $Neg(H)$, the number of negative eigenvalues of the operator
$H$ counted with their multiplicity.\ Our lower bounds match the well-known
upper bounds from CLR theory.

\section{Preliminaries}

\setcounter{equation}{0}

\subsection{Homogeneous ultrametric space}

Let $(X,d)$ be a locally compact and separable ultrametric space. Recall that
a metric $d$ is called an \emph{ultrametric} if it satisfies the ultrametric
inequality
\begin{equation}
d(x,y)\leq\max\{d(x,z),d(z,y)\},
\end{equation}
that is stronger than the usual triangle inequality. The basic consequence of
the ultrametric property is that each open ball is a closed set. Moreover,
each point $x$ of a ball $B$ can be regarded as its center, any two balls $A$
and $B$ either do not intersect or one is a subset of another etc. See e.g.
Section 1 in \cite{BendikovKrupski} and references therein. In this paper we
assume that the ultrametric space $(X,d)$ is not compact and that it is
\emph{proper}, i.e. each closed $d$-ball is a compact set.

Let $\mathcal{B}$ be the set of all open balls and $\mathcal{B}(x)\subset
\mathcal{B}$ the set of all balls centred at $x$. Notice that the set
$\mathcal{B}$ is atmost countable whereas $X$ by itself may well be
uncountable, e.g. $X=[0,+\infty\lbrack$ with $\mathcal{B}$ consisting of all
$p$-adic intervals as explained in the introduction.

To any ultrametric space $(X,d)$ one can associate in a standard fashion a
tree $\mathcal{T}.$ The vertices of the tree are metric balls, the boundary
$\partial\mathcal{T}$ can be identified with the one-point compactification
$X\cup\{\varpi\}$ of $X.$ We refere to \cite{BendikovKrupski} for a treatment
of the association between a ultrametric space and the tree of its metric balls.

A ultrametric measure space $(X,d,m)$ is called \emph{homogeneous} if the
group of isometries of $(X,d)$ acts transitively and preserves the measure. In
particular,\ a homogeneous ultrametric measure space is eather discrete or
perfect. In a homogeneous ultrametric measure space any two balls $A$ and $B$
having the same diameter satisfy $m(A)=m(B)$.

\subsection{Homogeneous hierarchical Laplacian}

\ Let $(X,d,m)$ be a homogeneous ultrametric space. Let $C:\mathcal{B}%
\rightarrow(0,\infty)$ be a function satisfying the following two conditions:
$(i)$ $C(A)=C(B)$ for any two balls $A$ and $B$ of the same diameter, and
$(ii)$ for all non-singletone $B\in\mathcal{B}$,%
\begin{equation}
\lambda(B):=\sum\limits_{T\in\mathcal{B}:\text{ }B\subseteq T}C(T)<\infty.
\label{C1 condition}%
\end{equation}
The class of functions $C(B)$ satisfying $(i)$ and $(ii)$ is reach enough,
e.g. one can choose
\[
C(B)=(1/m(B))^{\alpha}-(1/m(B^{\prime}))^{\alpha}%
\]
for any two closest neighboring balls $B\subset B^{\prime}$. In this case
\[
\lambda(B)=(1/m(B))^{\alpha}.
\]
Let $\mathcal{D}$ be the set of all locally constant functions having compact
support. The set $\mathcal{D}$ belongs to Banach spaces $C_{0}(X)$ and
$L^{p}=L^{p}(X,m),$ $1\leq p<\infty,$ and is a dence subset there. Given the
data $(\mathcal{B},C,m)$ we define (pointwise) \emph{the} \emph{homogeneous}
\emph{h}$\emph{ierarchical}$ \emph{Laplacian} $L$ as follows
\begin{equation}
Lf(x):=\sum\limits_{B\in\mathcal{B}(x)}C(B)\left(  f(x)-\frac{1}{m(B)}%
\int\limits_{B}fdm\right)  \text{.} \label{hlaplacian}%
\end{equation}
The operator $(L,\mathcal{D})$ acts in $L^{2},$ is symmetric and admits a
complete system of eigenfunctions%
\begin{equation}
f_{B}=\frac{\mathbf{1}_{B}}{m(B)}-\frac{\mathbf{1}_{B^{\prime}}}{m(B^{\prime
})}, \label{eigenfunction}%
\end{equation}
where the couple $B\subset B^{\prime}$ runs over all nearest neighboring balls
having positive measure. The eigenvalue corresponding to $f_{B}$ is
$\lambda(B^{\prime})$ defined at (\ref{C1 condition}),
\[
Lf_{B}(x)=\lambda(B^{\prime})f_{B}(x).
\]
Since the system of eigenfunctions is complete, we conclude that
$(L,\mathcal{D})$ is essentially self-adjoint operator.

\emph{The intrinsic ultrametric} $d_{\ast}(x,y)$ is defined as follows
\begin{equation}
d_{\ast}(x,y):=\left\{
\begin{array}
[c]{ccc}%
0 & \text{when} & x=y\\
1/\lambda(x\curlywedge y) & \text{when} & x\neq y
\end{array}
\right.  , \label{intrinsic ultrametric}%
\end{equation}
where $x\curlywedge y$ be the minimal ball containing both $x$ and $y$. In
particular, for any ball $B,$%
\begin{equation}
\lambda(B)=\frac{1}{\mathrm{diam}_{\ast}(B)}. \label{intrinsic diameter}%
\end{equation}
\emph{The spectral function} $\tau\rightarrow N(\tau),$ see equation
(\ref{Jump-kernel}), is defined as a left-continuous step-function having
jumps at the points $\lambda(B)$, and%
\[
N(\lambda(B))=1/m(B).
\]
$\emph{The}$ \emph{volume function }$V(r)$ is defined by setting $V(r)=m(B)$
where the ball $B$ has $d_{\ast}$-radius $r$. It is easy to see that
\begin{equation}
N(\tau)=1/V(1/\tau). \label{Spectral function}%
\end{equation}
The Markov semigroup $P_{t}=e^{-tL},t>0,$ admits a density $p(t,x,y)$ w.r.t.
$m$, we call it \emph{the heat kernel.} $p(t,x,y)$ is a continuous function
which can be represented in the form%
\begin{equation}
\emph{p}(t,x,y)=t\int\limits_{0}^{1/\emph{d}_{\ast}(x,y)}N(\tau)\exp
(-t\tau)d\tau. \label{Semigroup}%
\end{equation}
For $\lambda>0$ the resolvent operator $R_{\lambda}=(\lambda+L)^{-1}$ admits a
continuous strictly positive kernel $R(\lambda,x,y)$ with respect to the
measure $m$. The operator $R_{\lambda}$\ is well defined for $\lambda=0,$ i.e.
the Markov semigroup $(P_{t})_{t>0}$ is transient, if and only if for some
(equivalently, for all) $x\in X$ the function $\tau\rightarrow1/V(\tau)$ is
integrable at $\infty$. Its kernel $R(0,x,y)$, called also the Green function,
is of the form%
\begin{equation}
R(0,x,y)=%
{\displaystyle\int\limits_{\emph{d}_{\ast}(x,y)}^{+\infty}}
\frac{d\tau}{V(\tau)}. \label{Green function}%
\end{equation}
Under certain reasonable conditions the equation from above takes the form
\[
R(0,x,y)\asymp\frac{\emph{d}_{\ast}(x,y)}{V(\emph{d}_{\ast}(x,y))}.
\]

\subsection{An example}

Let $\Phi:\mathbb{R}_{+}\rightarrow\mathbb{R}_{+}$ be an increasing
homeomorphism. For any two nearest neighbouring balls $B\subset B^{\prime}$ we
define
\begin{equation}
C(B)=\Phi\left(  1/m(B)\right)  -\Phi\left(  1/m(B^{\prime})\right)  .
\label{An example}%
\end{equation}
Then the following properties hold:

\begin{description}
\item[(i)] $\lambda(B)=\Phi\left(  1/m(B)\right)  $,

\item[(ii)] \ $d_{\ast}(x,y)=1/\Phi\left(  1/m(x\curlywedge y)\right)  $,

\item[(iii)] $V(r)\leq1/\Phi^{-1}(1/r).$ Moreover, $V(r)\asymp1/\Phi
^{-1}(1/r)$ whenever both $\Phi$ and $\Phi^{-1}$ are doubling and
$m(B^{\prime})\leq cm(B)$ for some $c>0$ and all neighboring balls $B\subset
B^{\prime}$. In turn, this yields
\end{description}

\begin{equation}
p(t,x,y)\asymp t\cdot\min\left\{  \frac{1}{t}\Phi^{-1}\left(  \frac{1}%
{t}\right)  ,\frac{1}{m(x\curlywedge y)}\Phi\left(  \frac{1}{m(x\curlywedge
y)}\right)  \right\}  , \label{HK-Ex}%
\end{equation}
and%
\begin{equation}
p(t,x,x)\asymp\Phi^{-1}\left(  \frac{1}{t}\right)  \label{Spectral asympt.}%
\end{equation}
for all $t>0$ and $x,y\in X$.

\subsection{$L^{2}$-multipliers}

As a special case of the general construction consider $X=\mathbb{Q}_{p}$, the
ring of $p$-adic numbers equipped with its standard ultrametric
$d(x,y)=\left\vert x-y\right\vert _{p}$. Notice that the ultrametric spaces
$(\mathbb{Q}_{p},\mathrm{d})$ and $(\mathbb{[}0,\infty\mathbb{)},d)$ with
non-eucledian$\ d,$ as explained in the introduction, are isometric.

Let $m$ be the normed Haar measure on the Abelian group $\mathbb{Q}_{p},$
\ $L^{2}=L^{2}(\mathbb{Q}_{p},m)$ and $\mathcal{F}:f\rightarrow\widehat{f}$
the Fourier transform of function $f\in L^{2}$. It is known, see
\cite{Taibleson75}, \cite{Vladimirov94}, \cite{Kochubey2004}, that
$\mathcal{F}:\mathcal{D}\rightarrow\mathcal{D}$ is a bijection.

Let $\Phi:\mathbb{R}_{+}\rightarrow\mathbb{R}_{+}$ be an increasing
homeomorphism. The self-adjoint operator $\Phi(\mathfrak{D)}$ we define as
$L^{2}-$multiplier, that is,
\[
\widehat{\Phi(\mathfrak{D)}f}(\xi)=\Phi(\left\vert \xi\right\vert
_{p})\widehat{f}(\xi),\text{ \ }\xi\in\mathbb{Q}_{p}.
\]
By \cite[Theorem 3.1]{BGPW}, $\Phi(\mathfrak{D)}$ is a homogeneous
hierarchical Laplacian. The eigenvalues $\lambda(B)$\ of the operator
$\Phi(\mathfrak{D)}$ are of the form
\begin{equation}
\lambda(B)=\Phi\left(  \frac{p}{m(B)}\right)  . \label{Lambda-Phi eigenvalue}%
\end{equation}
\ Let $p(t,x,y)$ be the heat kernel associated with the operator
$\Phi(\mathfrak{D}).$ Assume that both $\Phi$ and $\Phi^{-1}$ are doubling,
then equations (\ref{HK-Ex}) and (\ref{Spectral asympt.}) apply. Since
$m(x\curlywedge y)=\left\vert x-y\right\vert _{p}$ we obtain
\begin{equation}
p(t,x,y)\asymp t\cdot\min\left\{  \frac{1}{t}\Phi^{-1}\left(  \frac{1}%
{t}\right)  ,\frac{1}{\left\vert x-y\right\vert _{p}}\Phi\left(  \frac
{1}{\left\vert x-y\right\vert _{p}}\right)  \right\}  , \label{HK-bounds II}%
\end{equation}
and%
\begin{equation}
p(t,x,x)\asymp\Phi^{-1}\left(  \frac{1}{t}\right)  . \label{J-bounds I}%
\end{equation}
The Taibleson-Vladimirov operator $\mathfrak{D}^{\alpha}$ is $L^{2}%
$-multiplier. By what we said above, its heat kernel $p_{\alpha}(t,x,y)$
satisfy
\begin{equation}
p_{\alpha}(t,x,y)\asymp\frac{t}{(t^{1/\alpha}+\left\vert x-y\right\vert
_{p})^{1+\alpha}},
\end{equation}%
\[
p_{\alpha}(t,x,x)\asymp\frac{1}{t^{1/\alpha}}.
\]
The Markov semigroup $(e^{-t\mathfrak{D}^{\alpha}})_{t>0}$ is transient if and
only if $\alpha<1$. In the transient case the Green function is of the form%
\begin{equation}
R_{\alpha}(0,x,y)=\frac{1-p^{-\alpha}}{1-p^{\alpha-1}}\frac{1}{\left\vert
x-y\right\vert _{p}^{1-\alpha}}.
\end{equation}
For all facts listed above we refere the reader to \cite{BGP}, \cite{BGPW} and
\cite{BendikovKrupski}.

\subsection{Schr\"{o}dinger operator}

Let $L$ be a homogeneous hierarchical Laplacian acting on $(X,d,m)$. Notice
that thanks to homogenuity $X$ can be identified with certain locally compact
Abelian group equipped with translation invariant ultrametric $d$ and the Haar
measure $m$. This identification is not unique (!) One possible way to define
such identification is to choose the sequence $\{a_{n}\}$\ of forward degrees
associated with the tree of balls $\Upsilon(X)$. This sequence is two-sided if
$X$ is non-compact and perfect, it is one-sided if $X$ is compact and perfect,
or if $X$ is discrete. In the 1st case we identify $X$ with $\Omega_{a}$, the
ring of $a$-adic numbers, in the 2nd case with $\Delta_{a}\subset\Omega_{a}$,
the ring of $a$-adic integers, and in the 3rd case with the factor group
$\Omega_{a}/\Delta_{a}$. We refere to \cite[(10.1)-(10.11), (25.1),
(25.2)]{HewittRoss} for the comprehensive treatment of special groups
$\Omega_{a}$, $\Delta_{a}$ and $\Omega_{a}/\Delta_{a}$.

By this identification $-L$ becomes a translation invariant isotropic Markov
generator, it acts as $L^{2}$-multiplicator as explained in the preceding
subsection. Yet, by (\ref{Spectrum}), the operator $(-L,\mathcal{D})$ can be
regarded as a symmetric L\'{e}vy generator%
\begin{equation}
-Lf(x)=%
{\displaystyle\int\limits_{X}}
(f(x+y)-f(x))J(y)dm(y)\text{ } \label{Levy generator}%
\end{equation}
where the measure $J(y)dm(y)$ (the L\'{e}vy measure associated to $-L$) is
finite on the complement $O^{c}$ of any open neighbourhood $O$\ of the neutral
element. Respectively, the semigroup $(e^{-tL})_{t>0}$ acts as a weakly
continuous convolution semigroup of probability measures $(\mu_{t})_{t>0}$.
Each measure $\mu_{t}$ is absolutely continuous with respect to $m$ and admits
a continuous symmetric density $\mu_{t}(x).$ In particular, the transition
density $p(t,x,y)$\ can be expressed in the form $p(t,x,y)=\mu_{t}(x-y).$ The
same of course true for the $\lambda$-Green function $R(\lambda,x,y)$, the
density of the resolvent $(L+\lambda)^{-1}$ with respect to $m$.

Consider the Schr\"{o}dinger operator $Hu(x)=Lu(x)+V(x)u(x)$ where $V(x)$ is a
real locally bounded measurable function.

\begin{theorem}
\label{Schroedinger spectrum}The following properties hold true:

$(i)$ The operator $H$ is essentially self-adjoint.\footnote{For the classical
Schr\"{o}dinger operator similar statement is known as Sears's theorem. It
holds true if the potential admits certain low bound and may fail otherwise,
see \cite[Chapter II, Theorem 1.1 and Example 1.1]{BeresinShubin}}

$(ii)$ Assume that $V(x)$ tends to plus infinity at infinity. Then the
operator $H$ has a compact resolvent, so that its spectrum is discrete.

$(iii)$ Assume that $V(x)$ tends to zero at infinity. Then the essential
spectrum of $H$ coincides with the essential spectrum of $L$. (Thus, $Spec(H)$
is pure point and the negative part of the spectrum of $H$ consists of
isolated eigenvalues of finite multiplicity).

$(iv)$ Assume that $V(x)$ tends to zero at infinity, the Markov semigroup
$(e^{-tL})_{t>0}$ is transient and that
\[
p(t,x,x)\asymp\Phi^{-1}(1/t)
\]
for some $\Phi$ as in (\ref{Spectral asympt.}) (see also (\ref{J-bounds I})).
Then the number $Neg(H)$ of negative eigenvalues counted with their
multiplicity is finite and satisfies%
\begin{equation}
Neg(H)\leq C\int_{X}\Phi^{-1}\circ V_{-}(x)dm(x) \label{Neg(H)_leq_1}%
\end{equation}
for some constant $C>0$ and $V_{-}(x)=-\min\{V(x),0\}.$
\end{theorem}

\begin{proof}
$(i)$ We follow the argument of \cite[Theorem 3.2]{Kochubey2004}. Let us
choose an open ball $O$ which contains the neutral element and write equation
(\ref{Levy generator}) in the form%
\begin{align*}
Lf(x)  &  =\left(
{\displaystyle\int\limits_{O}}
+%
{\displaystyle\int\limits_{O^{c}}}
\right)  [f(x)-f(x+y)]J(y)dm(y)\\
&  =L_{O}f(x)+L_{O^{c}}f(x).
\end{align*}
We have $Hf=L_{O}f+L_{O^{c}}f+Vf$, where the operator $V$ is the operator of
multiplication by the function $V(x)$. The operator $L_{O^{c}}f=J(O^{c}%
)(f-a\ast f)$, where $a(y)=J(y)1_{O^{c}}(y)/J(O^{c}),$ is bounded in
$L^{2}(X,m)$\ (as $f\rightarrow a\ast f$ \ is the operator of convolution with
probability measure $a(y)dm(y)$) and thus does not influence self-adjointness.
As $L_{O}$ is minus L\'{e}vy generator it is essentially self-adjoint (one
more way to make this conclusion is that the matrix of the operator $L_{O}%
$\ is diagonal in the basis $\{f_{B}\}$ of eigenfunctions of the operator $L$,
see \cite{Kozyrev}).

For any ball $B$ which belongs to the same horocycle $\mathcal{H}$ as $O$ we
denote $\mathfrak{H}_{B}$ the subspace of $L^{2}(X,m)$ which consists of all
functions $f$ \ having support in $B$. It is easy to see that $\mathfrak{H}%
_{B}$ is invariant with respect to symmetric operator $H_{O}=L_{O}+V$.
Moreover, $\mathfrak{H}_{B}$ reduces $H_{O}$.

The ultrametric space $X$ can be covered by a sequence of non-intersecting
balls $B_{n}$ (recall that due to the ultrametric property two balls of the
same diameter either coincide or do not itersect). This leads to the
orthogonal decomposition%
\[
L^{2}(X,m)=%
{\displaystyle\bigoplus\limits_{n}}
\mathfrak{H}_{B_{n}}%
\]
where each $\mathfrak{H}_{B_{n}}$ reduces $H_{O}$. The restriction of the
essentially self-adjoint operator $L_{O}$ to its invariant subspace
$\mathfrak{H}_{B_{n}}$ is an essentially self-adjoint operator, while the
restriction of the operator $V$ is bounded. Thus $H_{O}$ is essentially
self-adjoint as orthogonal sum of essentially self-adjoint operators $H_{O,n}%
$, the restriction of $H_{O}$ to $\mathfrak{H}_{B_{n}}$.

$(ii)$ The proof is similar to the one for the Schr\"{o}dinger operators given
in \cite[Theorem X.3]{Vladimirov94}; the main tools are boundness from below
of the operator $H$ and the analogues of the Riesz and Rellich compactness
criteria for subsets of $L^{2}(X,m)$.

$(iii)$ Let us show that the operator $V$ is $L-$\emph{compact. }Then, by
\cite[Theorem IV.5.35]{Kato}, the essential spectrums of the operators $H$ and
$L$ coinside. Recall that $L-$compactness means that if a sequence $\{u_{n}\}$
is such that both $\{u_{n}\}$ and $\{Lu_{n}\}$ are bounded then there exists a
subsequence $\{u_{n}^{\prime}\}\subset\{u_{n}\}$ such that the sequence
$\{Vu_{n}^{\prime}\}$ converges.

1. Denote $v_{n}=Lu_{n}+u_{n}.$ By assumption, the sequence $\{v_{n}\}$ is
bounded, and $u_{n}=r_{1}\ast v_{n}$. It follows that the quantity%
\[
\left(
{\displaystyle\int}
\left\vert u_{n}(x+h)-u_{n}(x)\right\vert ^{2}dm(x)\right)  ^{1/2}%
\leq\left\Vert v_{n}\right\Vert _{L^{2}}%
{\displaystyle\int}
\left\vert r_{1}(z+h)-r_{1}(z)\right\vert dm(z)
\]
tends to zero uniformly in $n$ as $h$ tends to the neutral element. Thus, the
sequence $\{u_{n}\}$ consists of equicontinuous on the whole in $L^{2}(X,m)$
functions. The same is true for the sequence $\{Vu_{n}\}$. Indeed, for any
ball $B$ which contains the neutral element we write%
\[
\left(
{\displaystyle\int}
\left\vert Vu_{n}(x+h)-Vu_{n}(x)\right\vert ^{2}dm(x)\right)  ^{1/2}\leq
I+II+III,
\]
where%
\[
I=\left\Vert V\right\Vert _{L^{\infty}}\left(
{\displaystyle\int}
\left\vert u_{n}(x+h)-u_{n}(x)\right\vert ^{2}dm(x)\right)  ^{1/2},
\]%
\[
II=\left\Vert u_{n}\right\Vert _{L^{2}}\left(
{\displaystyle\int_{B}}
\left\vert V(x+h)-V(x)\right\vert ^{2}dm(x)\right)  ^{1/2},
\]%
\[
III=\left\Vert u_{n}\right\Vert _{L^{2}}\sup_{x\in B^{c}}\left\vert
V(x+h)-V(x)\right\vert .
\]
Clearly $I,II$ and $III$ tend to zero uniformly in $n$ as $h$ tends to the
neutral element and $B\nearrow X$.

2. The sequence $\{Vu_{n}\}$ consists of functions with equicontinuous
$L^{2}(X,m)$ integrals at infinity. Indeed, for any ball $B$ which contains
the neutral element we have%
\[%
{\displaystyle\int\limits_{B^{c}}}
\left\vert Vu_{n}(x)\right\vert ^{2}dm(x)\leq\left\Vert u_{n}\right\Vert
_{L^{2}}\sup_{x\in B^{c}}\left\vert V(x)\right\vert \rightarrow0
\]
uniformly in $n$ as $B\nearrow X.$

Thus, the sequence $\{Vu_{n}\}$ is bounded in $L^{2}(X,m)$, consists of
equicontinuous on whole in $L^{2}(X,m)$ \ functions with equicontinuous
$L^{2}(X,m)$ integrals at infinity. By the Riesz-Kolmogorov criterion of
compactness in $L^{2}(X,m)$, the set $\{Vu_{n}\}$ is compact, whence it
contains a convergent subsequence $\{Vu_{n}^{\prime}\},$\ as claimed.

$(iv)$ By \cite[Theorem 2.1 and Remark 2.2]{MolchanovVainberg2}, the number of
negative eigenvalues counted with their multiplicity can be estimated as
follows:
\begin{equation}
Neg(H)\leq\frac{1}{c(\sigma)}\int_{X}V_{\_}(x)\left(  \int_{\frac{\sigma
}{V_{-}(x)}}^{\infty}p(t,x,x)dt\right)  dm(x) \label{MolchanovVainberg}%
\end{equation}
holds for any $\sigma>0$ with $c(\sigma)=e^{-\sigma}\int_{0}^{\infty
}z(z+\sigma)^{-1}e^{-z}dz$. Since we assume that $\Phi^{-1}$\ is doubling,
\[
\int_{\tau}^{\infty}p(t,x,x)dt\asymp\int_{\tau}^{\infty}\Phi^{-1}%
(1/t)dt\asymp\tau\Phi^{-1}(1/\tau)\text{ at }\infty.
\]
Choosing $\sigma$ big enough and applying the inequality
(\ref{MolchanovVainberg}) we obtain the inequality (\ref{Neg(H)_leq_1}), as desired.
\end{proof}

\section{Discrete ultrametric spaces}

Recall that a hierarchical Laplacian $L$ acting on a homogeneous ultrametric
measure space $(X,d,m)$\ is called \emph{homogeneous} if it is invariant under
the action of the group of isometries. By the homogenuity property for any two
balls $A$ and $B$ having the same diameter the eigenvalues $\lambda(A)$ and
$\lambda(B)$ coinside. We denote by $\lambda_{k}$ the common value of
eigenvalues for balls which belong to the horocycle $\mathcal{H}_{k}$. In this
section we assume that the ultrametric mesure space $(X,d,m)$ is
\emph{countably infinite and homogeneous}.

\subsection{Rank one perturbations}

Let us consider the Schr\"{o}dinger operator $H=L-V$ with potential
$V=\sigma\delta_{a}$.\ The operator $V:f(x)\rightarrow V(x)f(x)$ can be
written in the form%
\[
Vf(x)=\sigma(f,\delta_{a})\delta_{a}(x),
\]
i.e. $H$ can be regarded as rank one perturbation of the operator $L$.

Denote $\psi(x)=\mathcal{R}(\lambda,x,y),$ the solution of the equation%
\[
L\psi(x)-\lambda\psi(x)=\delta_{y}(x).
\]
Let $\psi_{V}(x)=\mathcal{R}_{V}(\lambda,x,y)$ be the solution of the equation%
\[
H\psi_{V}(x)-\lambda\psi_{V}(x)=\delta_{y}(x).
\]
Notice that $L$ and $H$ are symmetric operators whence both $(x,y)\rightarrow
\mathcal{R}(\lambda,x,y)$ and \ $(x,y)\rightarrow\mathcal{R}_{V}(\lambda,x,y)$
are symmetric functions.

\begin{proposition}
\label{Krein identity I}In the notation from above%
\begin{equation}
\mathcal{R}_{V}(\lambda,x,y)=\mathcal{R}(\lambda,x,y)+\frac{\sigma
\mathcal{R}(\lambda,x,a)\mathcal{R}(\lambda,a,y)}{1-\sigma\mathcal{R}%
(\lambda,a,a)}. \label{Krein}%
\end{equation}
In particular,
\begin{equation}
\mathcal{R}_{V}(\lambda,a,y)=\frac{\mathcal{R}(\lambda,a,y)}{1-\sigma
\mathcal{R}(\lambda,a,a)}. \label{Krein I}%
\end{equation}
and%
\begin{equation}
\mathcal{R}_{V}(\lambda,a,a)=\frac{\mathcal{R}(\lambda,a,a)}{1-\sigma
\mathcal{R}(\lambda,a,a)}. \label{Krein I'}%
\end{equation}

\end{proposition}

\begin{proof}
We have%
\begin{align*}
L\psi_{V}(x)-\lambda\psi_{V}(x)  &  =\delta_{y}(x)+\sigma\delta_{a}(x)\psi
_{V}(x)\\
&  =\delta_{y}(x)+\sigma\delta_{a}(x)\psi_{V}(a).
\end{align*}
It follows that%
\begin{equation}
\psi_{V}(x)=\mathcal{R}(\lambda,x,y)+\sigma\psi_{V}(a)\mathcal{R}%
(\lambda,x,a). \label{Krein I''}%
\end{equation}
Setting $x=a$ in the above equation we obtain%
\[
\psi_{V}(a)=\mathcal{R}(\lambda,a,y)+\sigma\psi_{V}(a)\mathcal{R}%
(\lambda,a,a)
\]
or%
\[
\psi_{V}(a)(1-\sigma\mathcal{R}(\lambda,a,a))=\mathcal{R}(\lambda,a,y).
\]
Since $\psi_{V}(a)=\mathcal{R}_{V}(\lambda,a,y)$ we obtain equation
(\ref{Krein I'}). In turn, equations (\ref{Krein I'}) and (\ref{Krein I''})
imply (\ref{Krein}) and (\ref{Krein I}).
\end{proof}

Let $\Upsilon(X)$ be the tree of balls associated with the ultrametric space
$(X,d)$.\ Consider in $\Upsilon(X)$ the infinite geodesic path from $a$ to
$\varpi:$\ $\{a\}=B_{0}\varsubsetneq B_{1}\varsubsetneq...\varsubsetneq
B_{k}\varsubsetneq...$ . The series below converges uniformly and in $L^{2},$%
\[
\delta_{a}=\left(  \frac{1_{B_{0}}}{m(B_{0})}-\frac{1_{B_{1}}}{m(B_{1}%
)}\right)  +\left(  \frac{1_{B_{1}}}{m(B_{1})}-\frac{1_{B_{2}}}{m(B_{2}%
)}\right)  +...=\sum_{k=0}^{\infty}f_{B_{k}}.
\]
Notice that all $f_{B_{k}}$ are eigenfunctions of the operator $L$, to be more
precise $Lf_{B_{k}}=\lambda(B_{k+1})f_{B_{k}}=\lambda_{k+1}f_{B_{k}}.$ Observe
that by definition $\mathcal{R}(\lambda,x,y)=(L-\lambda)^{-1}\delta_{y}(x)$
whence we obtain
\begin{align*}
\mathcal{R}(\lambda,a,a)  &  =\frac{1}{\lambda_{1}-\lambda}f_{B_{0}}%
(a)+\frac{1}{\lambda_{2}-\lambda}f_{B_{1}}(a)+...\\
&  =\frac{1}{\lambda_{1}-\lambda}\left(  \frac{1}{m(B_{0})}-\frac{1}{m(B_{1}%
)}\right) \\
&  +\frac{1}{\lambda_{2}-\lambda}\left(  \frac{1}{m(B_{1})}-\frac{1}{m(B_{2}%
)}\right)  +...\text{ ,}%
\end{align*}
or in the final form%
\begin{equation}
\mathcal{R}(\lambda,a,a)=%
{\displaystyle\sum\limits_{k=1}^{\infty}}
\frac{A_{k}}{\lambda_{k}-\lambda},\text{ \ }A_{k}=\left(  \frac{1}{m(B_{k-1}%
)}-\frac{1}{m(B_{k})}\right)  . \label{Resolvent poles}%
\end{equation}
Since $\lambda\rightarrow\mathcal{R}(\lambda,a,a)$ is an increasing function,
the equation%
\begin{equation}
1-\sigma\mathcal{R}(\lambda,a,a)=0,\text{ \ }\sigma\neq0,
\label{sigma-equation}%
\end{equation}
has precisely one solution $\lambda_{k}^{\sigma}$ lying in each open interval
$\left]  \lambda_{k+1},\lambda_{k}\right[  $ ,%
\[
\lambda_{k+1}<\lambda_{k}^{\sigma}<\lambda_{k}\text{, \ }k=1,2,...\text{ .}%
\]
\textbf{Claim 1} All numbers $\lambda_{k}^{\sigma}$ are eigenvalues of the
operator $H$. Indeed, the function $\psi(x)=\mathcal{R}(\lambda,x,a)$ with
$\lambda=$ $\lambda_{k}^{\sigma}$ satisfies the equation%
\begin{align*}
H\psi(x)-\lambda\psi(x)  &  =L\psi(x)-\lambda\psi(x)-\sigma\delta_{a}%
(x)\psi(x)\\
&  =L\psi(x)-\lambda\psi(x)-\sigma\delta_{a}(x)\psi(a)\\
&  =L\psi(x)-\lambda\psi(x)-\delta_{a}(x)=0.
\end{align*}
\textbf{Claim 2 }All numbers $\lambda_{k}$ are eigenvalues of the operator
$H$. Indeed, for any ball $B$ which does not contain $a$ but belongs to the
horocycle as $\mathcal{H}_{k-1}$ we have
\[
Hf_{B}=Lf_{B}=\lambda_{k}f_{B}.
\]
When $\sigma>0$ there may exist one more eigenvalue $\lambda_{-}^{\sigma}<0,$
a solution of the equation (\ref{sigma-equation}). Indeed, $\lambda
\rightarrow\mathcal{R}(\lambda,a,a)$ is an increasing function, continuous on
the interval $]-\infty,0]$. Since $\mathcal{R}(\lambda,a,a)\rightarrow0$ as
$\lambda\rightarrow-\infty$ and $\mathcal{R}(\lambda,a,a)\rightarrow
\mathcal{R}(0,a,a)\leq+\infty$ as $\lambda\rightarrow-0$, equation
(\ref{sigma-equation}) has a unique solution $\lambda=\lambda_{-}^{\sigma}<0$
in the following two cases:

\begin{description}
\item[$(i)$] The semigroup $(e^{-tL})_{t>0}$ is recurrent, i.e. $\mathcal{R}%
(0,a,a)=+\infty.$

\item[$(ii)$] The semigroup $(e^{-tL})_{t>0}$ is transient, i.e.
$\mathcal{R}(0,a,a)<+\infty,$ and $\sigma$ is such that $\mathcal{R}%
(0,a,a)>1/\sigma.$
\end{description}

Summarizing all the above we obtain the following result

\begin{proposition}
\label{Spectrum I} The operator $H=L-V$ with $V=$\ $\sigma\delta_{a}$ has at
most one negative eigenvalue and countably many positive eigenvalues with
accumulating point $0$. $\ $The operator $H$ has precisely one negative
eigenvalue $\lambda^{\sigma}$ if and only if $\sigma>0$ and one of the
conditions $(i)$ and $(ii)$ above holds. In this case the set $Spec(H)$
consists of points
\[
\lambda_{-}^{\sigma}<0<...<\lambda_{k+1}<\lambda_{k}^{\sigma}<\lambda
_{k}<...<\lambda_{2}<\lambda_{1}^{\sigma}<\lambda_{1}.
\]
Otherwise the set $Spec(H)$ consists of points%
\[
0<...<\lambda_{k+1}<\lambda_{k}^{\sigma}<\lambda_{k}<...<\lambda_{2}%
<\lambda_{1}^{\sigma}<\lambda_{1}%
\]
The eigenvalues $\lambda_{k}$ are at the same time eigenvalues of the operator
$L$. All $\lambda_{k}$ have infinite multiplicity and compactly supported
eigenfunctions, the eigenfunctions of the operator $L$ whose supports do not
contain $a$. The eigenvalue $\lambda_{k}^{\sigma}$ (resp. $\lambda_{-}%
^{\sigma}$) is the unique solution of the equation \ref{sigma-equation} in the
interval $\left]  \lambda_{k+1},\lambda_{k}\right[  $ (resp. in the interval
$]-\infty,0[$). $\lambda_{k}^{\sigma}$ (resp. $\lambda_{-}^{\sigma}$) has
multiplicity one and non-compactly supported eigenfunction $\psi
_{k}(x)=\mathcal{R}(\lambda_{k}^{\sigma},x,a)$ (resp. $\psi_{-}(x)=\mathcal{R}%
(\lambda_{-}^{\sigma},x,a)$).
\end{proposition}

\subsection{Finite rank perturbations}

Let us consider the Schr\"{o}dinger operator $H=L-V$ with potential
$V=\sum_{i=1}^{N}\sigma_{i}\delta_{a_{i}}$. The operator $V:f(x)\rightarrow
V(x)f(x)$ can be written in the form%
\[
Vf(x)=\sum_{i=1}^{N}\sigma_{i}(f,\delta_{a_{i}})\delta_{a_{i}}(x),
\]
i.e. $H$ can be regarded as rank $N$ perturbation of the operator $L$.

\begin{lemma}
\label{perturbation of finite rank}Let $A$ and $B$ be two symmetric operators,
\textrm{rank}$(B)=N$ and $H=A+B$. Let $(a,b)$ be an interval lying in the
complement of the set $Spec(A)$. The set $Spec(H)\cap(a,b)$ consists of at
most $N$ eigenvalues.
\end{lemma}

\begin{proof}
Assume that $N=1$. In this case $Bf=\sigma(f,f_{1})f_{1}$. For $\lambda
\in(a,b)$ any solution of the equation $Hf-\lambda f=0$ can be written in the
form $f=-\sigma(f,f_{1})R_{\lambda}f_{1}$ where $R_{\lambda}=(A-\lambda)^{-1}%
$. Taking inner product in both parts of this equation we get $(f,f_{1}%
)=-\sigma(f,f_{1})(R_{\lambda}f_{1},f_{1})$ or $\sigma(R_{\lambda}f_{1}%
,f_{1})+1=0$. Since the function $\lambda\rightarrow(R_{\lambda}f_{1},f_{1})$
is strictly increasing the equation $\sigma(R_{\lambda}f_{1},f_{1})+1=0$ has
at most one solution in the interval $(a,b)$. This solution is eigenvalue of
the operator $H$.

Assume that the statement holds true for $N=k$, and let $\lambda_{1}%
\leq...\leq\lambda_{k}$ be the corresponding eigenvalues in the interval
$(a,b)$. The numbers $\lambda_{i}$ split the interval $(a,b)$ in at most $k+1$
open intervals $I_{i}$ each of which does not intersect the spectrum of the
operator $A+\sum_{i=1}^{k}\sigma_{i}(f,f_{i})f_{i}$. Let us consider the
operator $H=A+\sum_{i=1}^{k+1}\sigma_{i}(f,f_{i})f_{i}$ and write
$H=A^{\prime}+\sigma_{k+1}(f,f_{k+1})f_{k+1}$. By what we have proved above
each of the $k+1$ open intervals $I_{i}$ contains at most one eigenvalue of
the operator $H$, i.e. $H$ contains in $(a,b)$ at most $k+1$ eigenvalues. The
proof is finished.
\end{proof}

\begin{example}
Consider the case: $\sigma_{i}=\sigma>0$ and $a_{i}$ fulfill the whole ball
$B_{0}\subset X$, i.e. $V(x)=\sigma1_{B_{0}}(x)$. Assume that $B_{0}$ belongs
to the horocycle $\mathcal{H}_{k_{0}}$. Let us select the following three
Hilbert subspaces of $L^{2}(X,m)$:
\end{example}

\begin{itemize}
\item $\mathcal{L}_{+}=\mathrm{span}\{1_{B}:B\in\mathcal{H}_{k_{0}}\}$, the
linear subspace of $L^{2}=L^{2}(X,m)$ spanned by the indicators of balls which
belong to the horocycle $\mathcal{H}_{k_{0}}$,

\item $\mathcal{L}_{-}=L^{2}(X,m)\ominus\mathcal{L}_{+}$, the orthogonal
complement of $\mathcal{L}_{+}$, and

\item $\mathcal{L}_{B}=\mathrm{span}\{f_{T}:T\varsubsetneq B\}$, the linear
space spanned by the eigenfunctions $f_{T}=1_{T}/m(T)-1_{T^{\prime}%
}/m(T^{\prime})$ of the operator $L$ such that $T^{\prime}\subseteq B.$
\end{itemize}

Let $\left\langle H|\mathcal{L}_{+}\right\rangle ,\left\langle H|\mathcal{L}%
_{-}\right\rangle $ and $\left\langle H|\mathcal{L}_{B}\right\rangle $) be the
restriction of the operator $H$ to its invariant subspaces listed above$.$

\textbf{Spectrum of the operator} $\left\langle H|\mathcal{L}_{-}\right\rangle
:$ The system of eigenfunctions $\{f_{T}:T\in\mathcal{B}\}$ is complete in
$L^{2}(X,m)$. Each eigenfunction $f_{T}$ either belongs to $\mathcal{L}_{+}$
or to its orthogonal complement $\mathcal{L}_{-}$ whence
\[
\mathcal{L}_{-}=%
{\displaystyle\bigoplus_{B\in\mathcal{H}_{k_{0}}}}
\mathcal{L}_{B}=\mathcal{L}_{B_{0}}%
{\displaystyle\bigoplus}
\mathcal{L}^{\prime}.
\]
For every ball $B\in\mathcal{H}_{k_{0}}$ we have
\[
\left\langle H|\mathcal{L}_{B}\right\rangle =\left\{
\begin{array}
[c]{cc}%
\left\langle L|\mathcal{L}_{B_{0}}\right\rangle -\sigma, & B=B_{0}\\
\left\langle L|\mathcal{L}_{B}\right\rangle , & B\neq B_{0}%
\end{array}
\right.
\]
whence
\[
\left\langle H|\mathcal{L}_{-}\right\rangle =(\left\langle L|\mathcal{L}%
_{B_{0}}\right\rangle -\sigma)%
{\displaystyle\bigoplus}
\left\langle L|\mathcal{L}^{\prime}\right\rangle
\]
and%
\[
\left\langle L|\mathcal{L}^{\prime}\right\rangle =%
{\displaystyle\bigoplus\limits_{B\neq B_{0}}}
\left\langle L|\mathcal{L}_{B}\right\rangle .
\]
It follows that%
\begin{equation}
Spec(\left\langle H|\mathcal{L}_{-}\right\rangle )=\{\lambda_{k}-\sigma:1\leq
k\leq k_{0}\}\cup\{\lambda_{k}:1\leq k\leq k_{0}\}.
\end{equation}
\textbf{The number }$Neg(\left\langle H|\mathcal{L}_{-}\right\rangle ):$ The
operator $\left\langle H|\mathcal{L}_{-}\right\rangle $ has finite number of
negative eigenvalues. Given $\sigma>\lambda_{k_{0}}$ this set is not empty.
Let us estimate $Neg(\left\langle H|\mathcal{L}_{-}\right\rangle )$, the
number of negative eigenvalues counted with their multiplicity. Assuming that
$\sigma>\lambda_{k_{0}}$\ let us choose the integer $k_{\ast}\leq k_{0}$ such
that%
\begin{equation}
\lambda_{k_{\ast}}<\sigma<\lambda_{k_{\ast}-1}. \label{Neg0}%
\end{equation}
Let $\{n_{k}\}$ be the sequence of forward degrees associated with the tree of
balls $\Upsilon(X)$. We have%
\begin{align*}
Neg(\left\langle H|\mathcal{L}_{-}\right\rangle )  &  =(n_{k_{0}}-1)+n_{k_{0}%
}(n_{k_{0}-1}-1)+...+n_{k_{0}}n_{k_{0}-1}...n_{k_{\ast}+1}(n_{k_{\ast}}-1)\\
&  =n_{k_{0}}\left(  1-\frac{1}{n_{k_{0}}}\right)  +...+\text{ }n_{k_{0}%
}n_{k_{0}-1}...n_{k_{\ast}+1}n_{k_{\ast}}\left(  1-\frac{1}{n_{k_{\ast}}%
}\right)  \text{.}%
\end{align*}
Since all $n_{k}\geq2$ we obtain%
\[
\frac{1}{2}\left(  n_{k_{0}}+...+\text{ }n_{k_{0}}n_{k_{0}-1}...n_{k_{\ast}%
}\right)  <Neg(\left\langle H|\mathcal{L}_{-}\right\rangle )<n_{k_{0}%
}+...+\text{ }n_{k_{0}}n_{k_{0}-1}...n_{k_{\ast}}%
\]
or by the same reasongs%
\begin{equation}
\frac{1}{2}n_{k_{0}}n_{k_{0}-1}...n_{k_{\ast}}<Neg(\left\langle H|\mathcal{L}%
_{-}\right\rangle )<\frac{3}{2}n_{k_{0}}n_{k_{0}-1}...n_{k_{\ast}}.
\label{Neg1}%
\end{equation}
Let $B_{\ast}\subseteq B_{0}$ be a ball in $\mathcal{H}_{k_{\ast}}$, then%
\[
n_{k_{0}}n_{k_{0}-1}...n_{k_{\ast}}=\frac{m(B_{0})}{m(B_{\ast})}%
\]
whence%
\begin{equation}
\frac{1}{2}\frac{m(B_{0})}{m(B_{\ast})}<Neg(\left\langle H|\mathcal{L}%
_{-}\right\rangle )<\frac{3}{2}\frac{m(B_{0})}{m(B_{\ast})}. \label{Neg2}%
\end{equation}
Let us choose an increasing function $\Phi$ such that the eigenvalues of the
operator $L$ can be written in the form $\lambda(B)=\Phi(1/m(B))$, then
equation (\ref{Neg0}) yields
\begin{equation}
\frac{1}{n_{k_{\ast}}}%
{\displaystyle\int\limits_{X}}
\Phi^{-1}\circ V(x)dm(x)<\frac{m(B_{0})}{m(B_{\ast})}<%
{\displaystyle\int\limits_{X}}
\Phi^{-1}\circ V(x)dm(x) \label{Neg3}%
\end{equation}
\textbf{Spectrum of the operator} $\left\langle H|\mathcal{L}_{+}\right\rangle
:$\textbf{ \ }We say that $x\sim y$ if and only if $x$ and $y$ belong to the
same ball $B\in\mathcal{H}_{k_{0}}$. Clearly this is an equivalence relation.
The set $[X]$ of all equivalence classes $[x]$\ equipped with the induced
metric and with the induced measure is a discrete homogeneous ultrametric
measure space. All balls $B\in\mathcal{H}_{k_{0}}$ become singletones $[B]$ in
$[X]$. Let $\tau:x\rightarrow\lbrack x]$ be the canonical mapping of $X$ onto
$[X].$ The mapping $\Upsilon:f\rightarrow f\circ\tau$ maps $L^{2}([X],[m])$
onto $\mathcal{L}_{+}$ isometrically. The operator $\left[  L\right]
=\Upsilon^{-1}\circ\left\langle L|\mathcal{L}_{+}\right\rangle \circ\Upsilon$
acting in $L^{2}([X],[m])$\ is a homogeneous hierarchical Laplacian. Since in
the formula%
\[
\left\langle L|\mathcal{L}_{+}\right\rangle f(x)=\sum\limits_{B\in
\mathcal{B}(x)}C(B)\left(  f(x)-\frac{1}{m(B)}\int\limits_{B}fdm\right)  ,
\]
the sum can be taken over all balls $B\in\mathcal{B}(x)$ each of which belongs
to the horocycle $\mathcal{H}_{k}$ with $k>k_{0}$, the complete list of
eigenvalues of the operator $[L]$ is: $\lambda_{k_{0}+1}>\lambda_{k_{0}%
+2}>\lambda_{k_{0}+3}>...$. The Markov semigroup $(e^{-t[L]})_{t>0}$ is
transient (resp. recurrent) if and only if the Markov semigroup $(e^{-tL}%
)_{t>0}$\ is transient (resp. recurrent).

The operator $\Upsilon^{-1}\circ\left\langle H|\mathcal{L}_{+}\right\rangle
\circ\Upsilon$ acting in $L^{2}([X],[m])$\ coincides with the Schr\"{o}dinger
operator $[H]=[L]-[V]$\ \ where $[V]=\sigma\cdot\delta_{\lbrack B_{0}]}.$ Thus
$Spec(\left\langle H|\mathcal{L}_{+}\right\rangle )=Spec([H])$ and the results
of Proposition \ref{Spectrum I} apply.

\textbf{Spectrum of the operator} $H:$\textbf{ \ }The operator $H=L-\sigma
1_{B_{0}}$\ has purely point spectrum. The set of positive eigenvalues of the
operator $H$ can be splittet in three subsets $\Xi_{1},\Xi_{2}$ and $\Xi_{3}$:
$\Xi_{1}$ consists of eigenvalues $\lambda_{k\text{ }}$of the operator $L$,
$\Xi_{2}=\{\lambda_{k\text{ }}-\sigma:k\leq k_{\ast}\}$ where $k_{\ast}%
=\max\{k\leq k_{0}:$ $\lambda_{k\text{ }}-\sigma>0\}$ and $\Xi_{3}$ consists
of eigenvalues $\lambda_{k}^{\sigma}$, $k\geq k_{0}$, each lying in the open
interval $(\lambda_{k+1},\lambda_{k})$. The set $\Xi_{-}$ of negative
eigenvalues of the operator $H$ consists of eigenvalues $\lambda_{k}-\sigma$,
$k_{\ast}<k\leq k_{0}$, and the eigenvalue $\lambda_{-}^{\sigma}$ described in
Proposition \ref{Spectrum I}. Each eigenvalue in $\Xi_{1}$ has infinite
multiplicity, the eigenfunctions in $\Xi_{2},\Xi_{3}$ and $\Xi_{-}$ have
finite multiplicity. The eigenfunctions corresponding to $\Xi_{1},\Xi_{2}$ and
$\Xi_{-}\backslash\{\lambda_{-}^{\sigma}\}$ are compactly supported, the
eigenfunctions corresponding $\Xi_{3}$ and $\lambda_{-}^{\sigma}$ have full support.

\textbf{The number }$Neg(H):$ The operator $H$ has finite number of negative
eigenvalues. Evidently this set is not empty if, for instance, $\sigma
>\lambda_{k_{0}}$. Assuming that $\sigma>\lambda_{k_{0}}$ we estimate
$Neg(H)$, the number of negative eigenvalues counted with their multiplicity:
\begin{equation}
Neg(H)>\frac{1}{2n_{k_{\ast}}}%
{\displaystyle\int\limits_{X}}
\Phi^{-1}\circ V(x)dm(x), \label{Neg(H)-phi1}%
\end{equation}
where $k_{\ast}$ is defined by equation (\ref{Neg0}), and%
\begin{equation}
Neg(H)<1+\frac{3}{2}%
{\displaystyle\int\limits_{X}}
\Phi^{-1}\circ V(x)dm(x). \label{Neg(H)-phi2}%
\end{equation}

\begin{proposition}
Assume that $V=\sum_{i=1}^{N}\sigma_{i}\delta_{a_{i}}$\ and that all
$\sigma_{i}>0$ are different. For any $\delta$ large enough there exists
$k(\delta)$ such that $\min_{i\neq j}d(a_{i},a_{j})>\delta$ implies that the
operator $H$ has precisely $N$ different eigenvalues in each open interval
$]\lambda_{k+1},\lambda_{k}[$: $1\leq k\leq k(\delta)$. Moreover there exists
precisely $N$ negative eigenvalues in the following two cases:
\end{proposition}

\begin{description}
\item[$(i)$] The semigroup $(e^{-tL})_{t>0}$ is recurrent.

\item[$(ii)$] The semigroup $(e^{-tL})_{t>0}$ is transient and $\sigma
_{i}>1/\mathcal{R}(0,a,a)$.
\end{description}

\begin{proof}
Let $\psi(x):=\mathcal{R}(\lambda,x,y)$ be the solution of the equation
$L\psi(x)-\lambda\psi(x)=\delta_{y}(x)$, and $\psi_{V}(x):=\mathcal{R}%
_{V}(\lambda,x,y)$ be the solution of the equation $H\psi_{V}(x)-\lambda
\psi_{V}(x)=\delta_{y}(x)$. To find the perturbed resolvent $\mathcal{R}%
_{V}(\lambda,x,y)$ we write
\begin{align*}
L\psi_{V}(x)-\lambda\psi_{V}(x)  &  =\delta_{y}(x)+\sum_{i=1}^{N}\sigma
_{j}\delta_{a_{j}}(x)\psi_{V}(x)\\
&  =\delta_{y}(x)+\sum_{j=1}^{N}\sigma_{j}\psi_{V}(a_{j})\delta_{a_{j}}(x),
\end{align*}
or%
\begin{equation}
\psi_{V}(x)=\mathcal{R}(\lambda,x,y)+\sum_{j=1}^{N}\sigma_{j}\psi_{V}%
(a_{j})\mathcal{R}(\lambda,x,a_{j})
\end{equation}
Choosing in the above equation $x=a_{1},x=a_{2},...,x=a_{N}$ we obtain system
of $N$ linear equations with $N$ variables $\xi_{i}=\psi_{V}(a_{i}),$%
\[
\xi_{i}=\mathcal{R}(\lambda,a_{i},y)+\sum_{j=1}^{N}\sigma_{j}\mathcal{R}%
(\lambda,a_{i},a_{j})\xi_{j},\text{ \ }i=1,2,...,N,
\]
or, in the vector form,%
\begin{equation}
(\mathrm{E}-\mathfrak{R}(\lambda)\Theta)\xi=\mathfrak{R}(\lambda,y),
\end{equation}
where we use the following notation $\xi=(\xi_{i})_{i=1}^{N},$ $\mathfrak{R}%
(\lambda,y)=(\mathcal{R}(\lambda,a_{i},y))_{i=1}^{N}$, \ $\mathfrak{R}%
(\lambda)=(\mathcal{R}(\lambda,a_{i},a_{j}))_{i,j=1}^{N}$, \textrm{E}%
$=(\delta_{ij})_{i,j=1}^{N}$ and \ $\Theta=\mathrm{diag}(\sigma_{i})$.

Let us substitute $\xi_{i}=\mathcal{R}_{V}(\lambda,a_{i},y)$ and
$\mathfrak{R}_{V}(\lambda,y):=(\mathcal{R}_{V}(\lambda,a_{i},y))_{i=1}^{N}$ in
the above equation
\begin{equation}
(\mathrm{E}-\mathfrak{R}(\lambda)\Theta)\mathfrak{R}_{V}(\lambda
,y)=\mathfrak{R}(\lambda,y). \label{Krein IV}%
\end{equation}
Choosing in equation (\ref{Krein IV}) $y=a_{1},y=a_{2},...,y=a_{N}$ and
setting $\mathfrak{R}_{V}(\lambda)=(\mathcal{R}_{V}(\lambda,a_{i}%
,a_{j}))_{i,j=1}^{N}$ we get the following matrix equation%
\[
(\mathrm{E}-\mathfrak{R}(\lambda)\Theta)\mathfrak{R}_{V}(\lambda
)=\mathfrak{R}(\lambda)
\]
or equivalently%
\[
(\mathrm{E}-\mathfrak{R}(\lambda)\Theta)\mathfrak{R}_{V}(\lambda
)\Theta=\mathfrak{R}(\lambda)\Theta.
\]
Similarly, since $H=L+\sum_{i=1}^{N}(-\sigma_{i})\delta_{a_{i}}$, we get%
\[
(\mathrm{E}+\mathfrak{R}_{V}(\lambda)(-\Theta))\mathfrak{R}(\lambda
)(-\Theta)=\mathfrak{R}_{V}(\lambda)(-\Theta).
\]
It follows that

$(i)$ the matrices $\mathfrak{R}(\lambda)\Theta$ and $\mathfrak{R}_{V}%
(\lambda)\Theta$ commute, and

$(ii)$ $\lambda\in Spec(H)\setminus Spec(L)$ if and only if $\lambda$
satisfies the equation%
\begin{equation}
\det(\mathrm{E}-\mathfrak{R}(\lambda)\Theta)=0. \label{ch. eq. I}%
\end{equation}
Observe that the variable $z:=\mathcal{R}(\lambda,a_{i},a_{i})$ does not
depend on $i,$ and its range is the whole interval $]-\infty,\infty\lbrack$
when $\lambda$ takes values in each of open interval $]\lambda_{k+1}%
,\lambda_{k}[.$ Equation (\ref{ch. eq. I}) can be written as characteristic
equation
\begin{equation}
\det(\mathfrak{A}-z\mathrm{E})=0 \label{ch. eq. II}%
\end{equation}
where $\mathfrak{A}=(\mathfrak{a}_{ij})_{i,j=1}^{N}$ is symmetric $N\times N$
matrix with entries $\mathfrak{a}_{ii}=1/\sigma_{i}$ and $\mathfrak{a}%
_{ij}=-\mathcal{R}(\lambda,a_{i},a_{j})$ for $i\neq j$. Thus all solutions of
equation (\ref{ch. eq. II}) are eigenvalues of the matrix $\mathfrak{A}$.

Let us compute $\mathcal{R}(\lambda,a_{i},a_{j})$. For any two neighboring
balls $B\subset B^{\prime}$ let us denote
\[
A(B)=\frac{1}{m(B)}-\frac{1}{m(B^{\prime})}.
\]
Let $a_{i}\curlywedge a_{j}$ be the minimal ball which contains both $a_{i}$
and $a_{j}$. Following the same line of reasongs as in the proof of equation
(\ref{Resolvent poles}) we obtain
\[
\mathcal{R}(\lambda,a_{i},a_{i})=%
{\displaystyle\sum\limits_{B:\text{ }a_{i}\in B}^{\infty}}
\frac{A(B)}{\lambda(B)-\lambda}%
\]
and
\[
\mathcal{R}(\lambda,a_{i},a_{j})=-\frac{m(a_{i}\curlywedge a_{j})^{-1}%
}{\lambda(a_{i}\curlywedge a_{j})-\lambda}+%
{\displaystyle\sum\limits_{B:\text{ }a_{i}\curlywedge a_{j}\subset B}}
\frac{A(B)}{\lambda(B)-\lambda}.
\]
Let us assume that $\min_{i\neq j}d(a_{i},a_{j})>\sigma>>1.$ Then $\max
\lambda(a_{i}\curlywedge a_{j})=\lambda_{k(\delta)+1}$ for some $k(\delta
)>>1$. For all $\lambda\geq\lambda_{k(\delta)}$ and all $i\neq j$ we get%
\[
\left\vert \mathcal{R}(\lambda,a_{i},a_{j})\right\vert <\frac{2m(a_{i}%
\curlywedge a_{j})^{-1}}{\lambda_{k(\delta)+1}-\lambda_{k(\delta)}}\leq
\frac{2\max\{m(a_{i}\curlywedge a_{j})^{-1}\}}{\lambda_{k(\delta)+1}%
-\lambda_{k(\delta)}}.
\]
Denote the left-hand side of the above inequality by $\varepsilon(\delta)/N$
and observe that this quantity tend to zero as $\delta\rightarrow\infty$. Let
us choose $\delta$ big enough so that the intervals $\{s:\left\vert
1/\sigma_{i}-s\right\vert \leq\varepsilon(\delta)\}$ do not intersect. By
Gershgorin Circle Theorem the matrix $\mathfrak{A}$ admits $N$ different
eigenvalues $\mathfrak{a}_{i}$ each of which lies in the corresponding open
interval $\{s:\left\vert 1/\sigma_{i}-s\right\vert <\varepsilon(\delta)\}.$
The eigenvalues $\mathfrak{a}_{i}$ are analytic functions of $\lambda$ in each
open interval $]\lambda_{k+1},\lambda_{k}[$, $1\leq k\leq k(\delta)$, see
\cite[Theorem XII.1]{ReedSimon}. Whence in each of these intervals the number
of different solutions of the equation $\mathfrak{a}_{i}=\mathcal{R}%
(\lambda,a_{i},a_{i})$ is at least $N$. Lemma
\ref{perturbation of finite rank} says that the number of different solutions
is at most $N$. Thus the number of different solutions is precisely $N$ as desired.
\end{proof}

\section{Perfect ultrametric spaces}

Let us consider a homogeneous ultrametric measure space $(X,d,m)$ which is
also assumed to be\emph{ non-compact}, \emph{perfect and proper}. Recall that
the last two properties mean: $X$ \ has no isolated points, and all closed
balls are compact sets.

\subsection{Potentials of finite rank}

Let $L$ be\ a homogeneous hierarchical Laplacian $L$. Let us fix a horocycle
$\mathcal{H}=\mathcal{H}_{k}$ and define $V(x)=\sum_{B\in\mathcal{H}}%
\sigma(B)1_{B}(x)$.\ We say that $V(x)$ is of\emph{ rank }$k$ \emph{function}.

Let $Hu(x)=Lu(x)-Vu(x)$ be the Schr\"{o}dinger operator with potential $V(x)$.
By Theorem \ref{Schroedinger spectrum}$(i)$, symmetric operator
$(H,\mathcal{D)}$ is essentially self-adjoint. Notice that (in contrary to the
case when the phase space is discrete) the operator $(L,\mathcal{D)}$ is
indeed \emph{unbounded} symmetric operator.

Among invariant subspaces of the operator $(H,\mathcal{D)}$ we select the
following three Hilbert spaces:

\begin{itemize}
\item $\mathcal{L}_{+}=\mathrm{span}\{1_{B}:B\in\mathcal{H}\}$, the linear
subspace of $L^{2}(X,m)$ spanned by the indicators of balls which belong to
the horocycle $\mathcal{H}$,

\item $\mathcal{L}_{-}=L^{2}\ominus\mathcal{L}_{+}$, the orthogonal complement
of $\mathcal{L}_{+}$, and

\item $\mathcal{L}_{B}=\mathrm{span}\{f_{T}:T\varsubsetneq B\}$, the linear
space spanned by the eigenfunctions $f_{T}=1_{T}/m(T)-1_{T^{\prime}%
}/m(T^{\prime})$ of the operator $L$ such that $T^{\prime}\subseteq B.$ The
space $\mathcal{L}_{B}$ is indeed the invariant subspace of $H.$
\end{itemize}

Each of the invariant subspaces spaces $\mathcal{L}_{+}$, $\mathcal{L}_{-}$
and $\mathcal{L}_{B}$ reduces the operators $L$ and $H$. We denote
$\left\langle L|\mathcal{L}_{+}\right\rangle $, $\left\langle H|\mathcal{L}%
_{+}\right\rangle $ etc. the restriction of the operators $L$ and $H$ to these
invariant subspaces.

\paragraph{Spectrum of the operator $\left\langle H|\mathcal{L}_{-}%
\right\rangle $}

The system of functions $\{f_{T}:T\in\mathcal{B}\}$ is complete in
$L^{2}(X,m)$. Since each $f_{T}$ either belongs to $\mathcal{L}_{+}$ or to its
orthogonal complement $\mathcal{L}_{-}$, we write
\[
\mathcal{L}_{-}=%
{\displaystyle\bigoplus\limits_{B\in\mathcal{H}}}
\mathcal{L}_{B}.
\]
For any ball $B\in\mathcal{H}$ we have
\[
\left\langle H|\mathcal{L}_{B}\right\rangle =\left\langle L|\mathcal{L}%
_{B}\right\rangle -\sigma(B),
\]
whence
\[
\left\langle H|\mathcal{L}_{-}\right\rangle =%
{\displaystyle\bigoplus\limits_{B\in\mathcal{H}}}
(\left\langle L|\mathcal{L}_{B}\right\rangle -\sigma(B)).
\]
Let us set $\mathfrak{S}(\mathcal{H})=Spec\left\langle L|\mathcal{L}%
_{B}\right\rangle $, then the above equation yield%
\begin{equation}
Spec\left\langle H|\mathcal{L}_{-}\right\rangle =%
{\displaystyle\bigcup\limits_{B\in\mathcal{H}}}
\{\mathfrak{S}(\mathcal{H})-\sigma(B)\}. \label{Spectrum on L_}%
\end{equation}

\paragraph{Spectrum of the operator $\left\langle H|\mathcal{L}_{+}%
\right\rangle $}

Consider the equivalence relation: $x\sim y$ if and only if $x$ and $y$ belong
to the same ball $B\in\mathcal{H}$. The set $[X]$ of all equivalence classes
\ equipped with the induced ultrametric $[d]$ and with the induced measure
$[m]$ is a discrete homogeneous ultrametric measure space$.$Without loss of
generality we may assume that $m(B)=1$, then the measure of all singletones in
$[X]$ is equal to $1$. Let $\tau:x\rightarrow\lbrack x]$ be the canonical
mapping of $X$ onto $[X].$ The mapping $\Upsilon:f\rightarrow f\circ\tau$ maps
$L^{2}([X],[m])$ onto $\mathcal{L}_{+}$ isometrically.

For any function $f(x)$ which takes constant values on the balls in
$\mathcal{H}$ we have%
\[
Lf(x)==\sum\limits_{[x]\subset B}C(B)\left(  f(x)-\frac{1}{m(B)}%
\int\limits_{B}fdm\right)
\]
where $[x]$ is the unique ball in $\mathcal{H}$\ (the equivalence class) which
contains $x$. The operator $\left[  L\right]  =\Upsilon^{-1}\circ\left\langle
L|\mathcal{L}_{+}\right\rangle \circ\Upsilon$ \ is a homogeneous hierarchical
Laplacian acting on the discrete homogeneous ultrametric measure space
$([X],[d],[m])$. The Markov semigroup $(e^{-t[L]})_{t>0}$ is transient (resp.
recurrent) if and only if the semigroup $(e^{-tL})_{t>0}$\ is transient (resp. recurrent).

The operator $\Upsilon^{-1}\circ\left\langle H|\mathcal{L}_{+}\right\rangle
\circ\Upsilon$ acting in $L^{2}([X],[m])$\ coincides with the Schr\"{o}dinger
operator $[H]=[L]-[V]$\ \ with potential $[V]=\sum\sigma(B)\cdot
\delta_{\lbrack B]}$. In particular, we obtain the following result%
\[
Spec\left\langle H|\mathcal{L}_{+}\right\rangle =Spec[H].
\]
Summing all the above we come to the following conclusion

\begin{proposition}
\label{Spectrum+}Assume that $V(x)\rightarrow0$ as $x\rightarrow$ $\infty$.
Then the spectrum of the operator $H=L-V$ is pure point with possibly
accumulating points $\lambda_{k}$, the eigenvalues of the operator $L$.

Assume that $\lambda(B)=\Phi(1/m(B))$ for some increasing homeomorphism
$\Phi:\mathbb{R}_{+}\rightarrow\mathbb{R}_{+}$, and for all balls $B$, then
\[
Neg(H)\geq C_{1}%
{\displaystyle\int_{\{x:\text{ }V(x)>\lambda_{\mathcal{\ast}}\}}}
\Phi^{-1}\circ V(x)dm(x)
\]
for some constant $C_{1}>0$, and for $\lambda_{\mathcal{\ast}}-$ the
eigenvalue corresponding to balls in $\mathcal{H}$.

Assume that the Markov semigroup $(e^{-tL})_{t>0}$ is transient and that both
$\Phi$ and $\Phi^{-1}$ are doubling, then%
\[
Neg(H)\leq C_{2}%
{\displaystyle\int_{X}}
\Phi^{-1}\circ V_{+}(x)dm(x)
\]
for some constant $C_{2}$, and $V_{+}(x)=\max\{0,V(x)\}$.
\end{proposition}

\begin{proof}
Negative eigenvalues $E_{i}$ of $H$ lie below the essential spectrum of $H$.
Hence to compute $E_{i}$ the $\min-\max$ principle applies, see \cite[Theorem
XIII.1]{ReedSimon}. In particular, $E_{i}$ depend monotonically on the
potential $V(x)$, whence without loss of generality one can assume that
$V(x)\geq0$.

As $B\downarrow\{a\}$ the eigenvalue $\lambda(B)=1/\mathrm{diam}_{\ast}(B)$
tend to $\infty$. It follows that the set of negative eigenvalues of the
operator $\left\langle H|\mathcal{L}_{B}\right\rangle $ is finite. In turn,
since $V(x)\rightarrow0$, the set of negative eigenvalues of the operator
$\left\langle H|\mathcal{L}_{-}\right\rangle $ is finite. Hence the proof of
the statement reduces to the discrete setting: $Spec\left\langle
H|\mathcal{L}_{+}\right\rangle =Spec[H]$. Thus, we apply Theorem
\ref{Schroedinger spectrum}$(ii)$ to get the first part of our claim. The
second part of the statement follows from equation (\ref{Neg(H)-phi1}). The
Markov semigroup $(e^{-tL})_{t>0}$ is transient and admits a continuous
transition density $p(t,x,y)$ such that $p(t,x,x)\asymp\Phi^{-1}(1/t)$. Since
we assume that $\Phi^{-1}$ is doubling,
\[
\int_{\tau}^{\infty}p(t,x,x)dt\asymp\tau\Phi^{-1}(1/\tau)\text{ at }\infty
\]
and%
\[
\int_{0}^{1}t^{m}p(t,x,x)dt<\infty\text{ \ for some }m>0.
\]
Thus, inequality (\ref{Neg(H)_leq_1}) applies and we come to the desired
conclusion
\[
Neg(H)\leq C_{2}%
{\displaystyle\int_{X}}
\Phi^{-1}\circ V_{+}(x)dm(x).
\]
The proof is finished.
\end{proof}

\subsection{The operator $H=\mathfrak{D}^{\alpha}-\sigma1_{B}$}

As an example let us consider $X=\mathbb{Q}_{p}$, the ring of $p$-adic numbers
and $\mathfrak{D}^{\alpha}$ the operator of fractional derivative of order
$\alpha$. Choose $\sigma>0$ and $B=\mathbb{Z}_{p},$ the set of $p$-adic
integers. Let $H=\mathfrak{D}^{\alpha}-V$ \ be the Schr\"{o}dinger operator
with potential $V=\sigma1_{B}$.

The eigenvalues of the operator $\mathfrak{D}^{\alpha}$ are numbers%
\[
\lambda(B)=\left(  \frac{p}{m(B)}\right)  ^{\alpha}=p^{-\alpha(k-1)}\text{,
\ }k\in\mathbb{Z},
\]
therefore the operator $\left\langle H|\mathcal{L}_{+}\right\rangle $ has
eigenvalues $\lambda_{k}=p^{-\alpha(k-1)}$, $k=1,2,...$.

Let $B_{0}\subset B_{1}\subset...$ be the infinite geodesic path in the
homogeneous (self-similar) tree $\Upsilon(\mathbb{Q}_{p})$ starting at
$B_{0}=B$ and ending at $\varpi.$ The operator $\left\langle H|\mathcal{L}%
_{+}\right\rangle $ we identify with operator $[H]$ acting on the discrete
lattice $\mathbb{Q}_{p}/\mathbb{Z}_{p}$ which can be identified with the set
$\mathbb{N}$ of integers equipped with the family of $p$-adic partitians (a
discrete counterpart of the Dyson's model)$.$ Let us compute the resolvent
$\mathcal{R}(\lambda,[B],[B])$ of the operator $[H]$\ at $\lambda=0$.
Following our computations in Section I we obtain
\[
\mathcal{R}(\lambda,[B],[B])=%
{\displaystyle\sum\limits_{k\geq1}}
\frac{A_{k}}{\lambda_{k}-\lambda}=(p-1)%
{\displaystyle\sum\limits_{k\geq1}}
\frac{1}{p^{k}(\lambda_{k}-\lambda)}.
\]
In particular, $\mathcal{R}(0,[B],[B])=+\infty$ if $\alpha\geq1$, otherwise
\[
\mathcal{R}(0,[B],[B])=\frac{p-1}{p}%
{\displaystyle\sum\limits_{k\geq0}}
\frac{1}{p^{k(1-\alpha)}}=\frac{p-1}{p-p^{\alpha}}.
\]
By Proposition \ref{Spectrum I}, the operator $\left\langle H|\mathcal{L}%
_{+}\right\rangle $ has atmost one negative eigenvalue. It does have a
negative eigenvalue if and only if either $(i)$ $\alpha\geq1$ or $(ii)$
$0<\alpha<1$ and $\sigma>(p-p^{\alpha})(p-1)^{-1}.$

By equation (\ref{Spectrum on L_}), the operator $\left\langle H|\mathcal{L}%
_{-}\right\rangle $ has atmost finite number of negative eigenvalues.
Evidently this set is not empty if $\sigma>\lambda(B)=p^{\alpha}$. To estimate
$Neg\left\langle H|\mathcal{L}_{-}\right\rangle $, the number of negative
eigenvalues counted with their multiplicity, we choose the integer $k_{\ast
}\geq0$ such that%
\begin{equation}
p^{k_{\ast}}\leq\frac{\sigma^{1/\alpha}}{p}<p^{k_{\ast}+1}. \label{k-star}%
\end{equation}
Let $B_{\ast}\subseteq B$ be a ball such that $m(B_{\ast})=p^{-k_{\ast}}$.
According to our choice $\lambda(B_{\ast})$ is the minimal eigenvalue
satisfying $\lambda(B)\leq\lambda(T)<\sigma$. Equations (\ref{Neg2}) and
(\ref{Neg3}) yield%
\[
\frac{1}{2}\frac{m(B)}{m(B_{\ast})}<Neg\left\langle H|\mathcal{L}%
_{-}\right\rangle <\frac{3}{2}\frac{m(B)}{m(B_{\ast})}%
\]
and
\[
\frac{1}{p}%
{\displaystyle\int}
V(x)^{1/\alpha}dm(x)<\frac{m(B)}{m(B_{\ast})}<%
{\displaystyle\int}
V(x)^{1/\alpha}dm(x).
\]
Let us define three subsets of the set $\{(\alpha,\sigma):\alpha
>0,\sigma>0\}:$

\begin{itemize}
\item $\Omega_{1}=\{(\alpha,\sigma):\sigma\leq(p-p^{\alpha})(p-1)^{-1},$

\item $\Omega_{2}=\{(\alpha,\sigma):(p-p^{\alpha})(p-1)^{-1}<\sigma\leq
p^{\alpha}\}$, \ 

\item $\Omega_{3}=\{(\alpha,\sigma):\sigma>p^{\alpha}\}.$
\end{itemize}

Let $Neg(H)$ be the number of negative eigenvalues of the operator $H$ counted
with their multiplicity. Summing all the above we conclude that%
\[
Neg(H)=\left\{
\begin{array}
[c]{ccc}%
0 & \text{if} & (\alpha,\sigma)\in\Omega_{1}\\
1 & \text{if} & (\alpha,\sigma)\in\Omega_{2}%
\end{array}
\right.
\]
and, if $(\alpha,\sigma)\in\Omega_{3}$, then
\[
\frac{1}{2p}%
{\displaystyle\int}
V(x)^{1/\alpha}dm(x)\leq Neg(H)\leq\frac{3}{2}%
{\displaystyle\int}
V(x)^{1/\alpha}dm(x).
\]

\section{The Dyson's dyadic model}

Let us consider $X=\mathbb{\{}0,1,2,...\mathbb{\}}$ equipped with the counting
measure $m.$ The set $\{\Pi_{r}:r=0,1,...\}$ of partitians of $X$ each of
which consists of dyadic intervals $I_{r}=\{l\in X:k2^{r}\leq l<(k+1)2^{r}\}$
induces in the standard way the ultrametric structure - a discrete version of
Dyson's model, as explained in the introduction. We call $r$ the rank of the
interval $I_{r}$. We denote $I_{r}(x)$\ the interval of rank $r$ which
contains the point $x$. Recall that the set of balls in the metric space $X$
coincides with the set of all dyadic intervals.

\bigskip

\textbf{5.1 Potentials of infinite rank }We consider the Schr\"{o}dinger
operator $H=L+V$ with bounded potential $V$ of the form $V(x)=\sum\sigma
_{k}1_{B_{k}}(x)$, the sequence of balls $B_{k}$ is chosen such that the rank
of $B_{k}$ tends to infinity. The choice of the potential $V$ will allow us to
conclude that $Spec_{sc}(H)$, the singular continuous part of the spectrum of
$H$, is not empty set. Here $L$ is the Dyson's hierarchical Laplacian%
\[
Lf(x)=%
{\displaystyle\sum\limits_{r=1}^{\infty}}
(1-\varkappa)\varkappa^{r}\left(  f(x)-\frac{1}{m(I_{r}(x))}%
{\displaystyle\int\limits_{I_{r}(x)}}
fdm\right)  ,
\]
with $\varkappa\in]0,1[$ a fixed parameter. Writing $\varkappa=2^{-\alpha}$ we
see that $L$ coinsides with the hierarchical Laplacian introduced in
(\ref{An example}) with $\Phi(\tau)=\tau^{\alpha}.$ In particular, all
eigenvalues $\lambda(B)$ of the operator $L$ are of the form
\[
\lambda(B)=\varkappa^{r}=|B|^{-\alpha},
\]
the natural number $r\geq1$ is the rank of the ball $B$ and $|B|$ its cardinality.

To define the potential $V(x)$ we choose a sequence of balls $B_{0}=\{0,1\}$,
$B_{1}=\{2,3\}$, $B_{2}=\{4,5,6,7\}$,..., $B_{k}=\{2^{k},...,2^{k+1}-1\},...$
and a sequence $\sigma_{0},\sigma_{1},\sigma_{2},...,\sigma_{k},...$ of
negative reals such that $\sigma_{0}\neq\sigma_{1}$, and set%
\[
V(x)=\sum\sigma_{k}1_{B_{k}}(x).
\]

We select the folowing $H-$ invariant subspaces of $L^{2}(X,m)$

\begin{itemize}
\item $\mathcal{L}_{+}=\mathrm{span}\{1_{B_{k}}:k=0,1,2,...\}$,

\item $\mathcal{L}_{B_{k}}=\mathrm{span}\{f_{T}:T\varsubsetneq B_{k}\}$, and

\item $\mathcal{L}_{-}=L^{2}(X,m)\ominus\mathcal{L}_{+}=%
{\displaystyle\bigoplus}
\mathcal{L}_{B_{k}}$.
\end{itemize}

\textbf{5.2 Spectrum of }$\left\langle H|\mathcal{L}_{-}\right\rangle
$\textbf{ }As in the previous section we conclude that

\begin{itemize}
\item \ $\left\langle H|\mathcal{L}_{B_{k}}\right\rangle =\left\langle
L|\mathcal{L}_{B_{k}}\right\rangle +\sigma_{k}$,

\item \ $\left\langle H|\mathcal{L}_{-}\right\rangle =%
{\displaystyle\bigoplus}
(\left\langle L|\mathcal{L}_{B_{k}}\right\rangle +\sigma_{k})$, and

\item $Spec\left\langle H|\mathcal{L}_{-}\right\rangle =\overline{%
{\displaystyle\bigcup}
\{\mathfrak{S}_{k}+\sigma_{k}\}}$ where $\mathfrak{S}_{k}=Spec\left\langle
L|\mathcal{L}_{B_{k}}\right\rangle .$
\end{itemize}

Let $B_{k}^{\prime}$ be the closest neighbouring ball to $B_{k}$. Since the
operator $L$ is homogeneous, $\mathfrak{S}_{k}=Spec\left\langle L|\mathcal{L}%
_{B_{k}^{\prime}}\right\rangle $. The sequence $\{B_{k}^{\prime}\}$ monotone
increase to $X$ whence $\mathfrak{S}_{k}\uparrow Spec(L)$. In particular, the
following statement holds true

\begin{proposition}
\label{Spec_ess}If the sequence $\{\sigma_{k}\}$ forms a dense subset in some
interval $\mathcal{I}$\ then the operator\emph{ }$\left\langle H|\mathcal{L}%
_{-}\right\rangle $ has a pure point spectrum. It contains all intervals\emph{
}$\tau+\mathcal{I}$, where\emph{ }$\tau$ runs over the set $Spec(L)$\emph{.
}In particular, the set $Spec\left\langle H|\mathcal{L}_{-}\right\rangle $
consists of finite number of disjoint intervals. \emph{ }
\end{proposition}

\textbf{5.3 Spectrum of }$<H|\mathcal{L}_{+}>$\textbf{ \ }As $f_{B_{k}}\perp$
$\mathcal{L}_{B_{k}}$ the sequence $\{f_{B_{k}}\}\subset\mathcal{L}_{+}$. We
claim that $\{f_{B_{k}}:k\geq1\}$ is a $\emph{complete}$ $\emph{orthogonal}$
\emph{system} in the Hilbert space $\mathcal{L}_{+}$. Indeed, when $k,l\geq1$
and $k\neq l$ we have $f_{B_{k}}\perp$ $f_{B_{l}}$ because $B_{k}^{\prime}\cap
B_{l}^{\prime}=\oslash$. Let us show that $\{f_{B_{k}}:k\geq1\}$\ is complete.
Assume that $\psi=\sum_{k\geq0}\psi_{k}1_{B_{k}}$ is orthogonal to all
$f_{B_{k}}$, $k=1,2,...$. Since $f_{B_{0}}+f_{B_{1}}=0$, \ $\psi$ is
orthogonal to $f_{B_{0}}$ as well. Thus, we obtain an infinite system of
linear equations $\sum_{l\geq0}\psi_{l}(1_{B_{l}},f_{B_{k}})=0$,
$k=0,1,2,...$, or in equivalent form,%
\begin{align*}
|B_{k}^{^{\prime}}|\psi_{k}  &  =\sum_{l}\psi_{l}(1_{B_{l}},1_{B_{k}^{\prime}%
})\\
&  =\sum_{l\leq k}\psi_{l}(1_{B_{l}},1_{B_{k}^{\prime}})=\sum_{l\leq k}%
|B_{l}|\psi_{l}.
\end{align*}
Setting $\xi_{l}=|B_{l}|\psi_{l}$, we obtain an infinite system of linear
equations $2\xi_{k}=\sum_{l\leq k}\xi_{l}$ which has unique solution $\xi
_{k}=0$, $k=0,1,2,...$, as claimed.

Thus, the system of functions $F_{k}=\sqrt{2|B_{k}|}f_{B_{k}},$ $k\geq1,$ is
an orthnormal basis in $\mathcal{L}_{+}$. The matrix $M_{L}$ of the operator
$<L|\mathcal{L}_{+}>$ in the basis $\{F_{k}\}$ is diagonal $M_{L}%
=\mathsf{diag}\{\varkappa^{2},\varkappa^{3},...\}.$ In particular, $L$ belongs
to the trace class, therefore by \cite[Theorem IV.5.35 and Theorem
X.4.4]{Kato},
\[
Spec_{ess}<H|\mathcal{L}_{+}>=Spec_{ess}<V|\mathcal{L}_{+}>
\]
and%
\[
Spec_{ac}<H|\mathcal{L}_{+}>=Spec_{ac}<V|\mathcal{L}_{+}>.
\]

The system of functions $E_{k+1}=|B_{k}|^{-1/2}I_{B_{k}},k\geq0,$ is an
orthnormal basis in $\mathcal{L}_{+}$. The matrix $M_{V}$ of the operator
$<V|\mathcal{L}_{+}>$ in the basis $\{E_{k}\}$ is diagonal $M_{V}%
=\mathsf{diag}\{\sigma_{0},\sigma_{1},\sigma_{2},\sigma_{3},...\}.$ In
particular, we conclude that:

\begin{itemize}
\item $Spec<V|\mathcal{L}_{+}>=\overline{\{\sigma_{n}:n=0,1,...\}},$

\item $Spec_{ac}<V|\mathcal{L}_{+}>$ and $Spec_{sc}<V|\mathcal{L}_{+}>$ are
empty sets,

\item $Spec_{d}<V|\mathcal{L}_{+}>$ consists of isolated $\sigma_{k}$ having
finite multiplicity,

\item $Spec_{ess}<V|\mathcal{L}_{+}>$ consists of isolated $\sigma_{k}$ having
infinite multiplicity and of occumulating points of the sequence $\{\sigma
_{k}\}$.
\end{itemize}

Summing up all fact from above we obtain

\begin{proposition}
\label{Spec_ess+}The set $Spec_{ess}<H|\mathcal{L}_{+}>$ consists of isolated
$\sigma_{k}$ having infinite multiplicity and of occumulating points of the
sequence $\{\sigma_{k}\}$, its subset $Spec_{ac}<H|\mathcal{L}_{+}%
>=\varnothing$. In particular, if the sequence $\{\sigma_{k}\}$ forms a dense
subset in some interval $\mathcal{I}$\ then $Spec_{ess}\left\langle
H|\mathcal{L}_{+}\right\rangle $\emph{= }$\mathcal{I}$\emph{.}
\end{proposition}

\textbf{5.4 Generalized eigenfunctions }Thus, we are left to find
$Spec_{pp}\left\langle H|\mathcal{L}_{+}\right\rangle $ and its subset
$Spec_{d}\left\langle H|\mathcal{L}_{+}\right\rangle $. Let us consider the
equation $H\psi=\lambda\psi$. We are looking for a bounded solution $\psi$ of
the form $\psi=\sum_{k}\psi_{k}1_{B_{k}}$ which satisfies the equation
$H\psi=\lambda\psi$ in a weak sense, that is,
\[
(H\phi-\lambda\phi,\psi)=0\text{ or }(L\phi,\psi)=((\lambda-V)\phi,\psi),
\]
for any test function $\phi\in\mathcal{D}$ of the form $\phi=\sum_{k}\phi
_{k}1_{B_{k}}$.

Let us choose $\phi=1_{B_{n}}$. As $V=\sum\sigma_{k}1_{B_{k}}$ and
$Lf_{B}=\lambda(B^{\prime})f_{B}$, we obtain:
\begin{align}
(\lambda-\sigma_{n})\psi_{n}  &  =(\psi,L(1_{B_{n}}/\left\vert B_{n}%
\right\vert ))=(\psi,L(f_{B_{n}}+f_{B_{n}^{\prime}}+f_{B_{n}^{\prime\prime}%
}+...))\label{theta eq}\\
&  =(\psi,\lambda_{n+1}f_{B_{n}}+\lambda_{n+2}f_{B_{n}^{\prime}}+\lambda
_{n+3}f_{B_{n}^{\prime\prime}}+...).\nonumber
\end{align}
Since $f_{B_{n}^{\prime}}=$ $-f_{B_{n+1}}$, $f_{B_{n}^{\prime\prime}%
}=-f_{B_{n+2}}$, etc equation (\ref{theta eq}) gives for $n\geq1$,
\begin{equation}
(\lambda-\sigma_{n})\psi_{n}=\lambda_{n+1}(\psi,f_{B_{n}})-\lambda_{n+2}%
(\psi,f_{B_{n+1}})-\lambda_{n+3}(\psi,f_{B_{n+2}})+..., \label{psi-equation}%
\end{equation}
and
\[
(\lambda-\sigma_{0})\psi_{0}=-\lambda_{2}(\psi,f_{B_{1}})-\lambda_{3}%
(\psi,f_{B_{2}})-\lambda_{4}(\psi,f_{B_{3}})-...\text{ .}%
\]
It follows that for $n\geq1$,
\begin{equation}
(\lambda-\sigma_{n})\psi_{n}-(\lambda-\sigma_{n+1})\psi_{n+1}=\lambda
_{n+1}(\psi,f_{B_{n}})-2\lambda_{n+2}(\psi,f_{B_{n+1}}),\text{ \ \ }
\label{sigma-psi eq}%
\end{equation}
and
\begin{equation}
(\lambda-\sigma_{0})\psi_{0}-(\lambda-\sigma_{1})\psi_{1}=-2\lambda_{2}%
(\psi,f_{B_{1}}). \label{initial cond_1}%
\end{equation}
Next we compute $\lambda_{n+1}(\psi,f_{B_{n}}):$
\begin{align*}
\lambda_{2}(\psi,f_{B_{1}})  &  =\frac{\lambda_{2}}{|B_{1}|}(\psi,1_{B_{1}%
})-\frac{\lambda_{2}}{|B_{1}^{\prime}|}(\psi,1_{B_{1}^{\prime}})\\
&  =\frac{1}{2}\lambda_{2}\left(  \psi_{1}-\psi_{0}\right)  ,
\end{align*}%
\begin{align*}
\lambda_{3}(\psi,f_{B_{2}})  &  =\frac{\lambda_{3}}{|B_{2}|}(\psi,1_{B_{2}%
})-\frac{\lambda_{3}}{|B_{2}^{\prime}|}(\psi,1_{B_{2}^{\prime}})\\
&  =\frac{1}{2}\lambda_{3}\left(  \psi_{2}-\frac{1}{2}(\psi_{1}+\psi
_{0})\right)
\end{align*}
and, for $n\geq3,$%
\begin{align*}
\lambda_{n+1}(\psi,f_{B_{n}})  &  =\frac{\lambda_{n+1}}{|B_{n}|}(\psi
,1_{B_{n}})-\frac{\lambda_{n+1}}{|B_{n}^{\prime}|}(\psi,1_{B_{n}^{\prime}})\\
&  =\frac{1}{2}\lambda_{n+1}\left(  \psi_{n}-\frac{1}{2}\psi_{n-1}%
-...-\frac{1}{2^{n-2}}\psi_{2}-\frac{1}{2^{n-1}}(\psi_{1}+\psi_{0})\right)  .
\end{align*}
Equations (\ref{initial cond_1}), (\ref{sigma-psi eq}) and computations from
above yield
\begin{equation}
\left(  \lambda-\sigma_{0}-\lambda_{2}\right)  \psi_{0}=\left(  \lambda
-\sigma_{1}-\lambda_{2}\right)  \psi_{1}, \label{initial condition}%
\end{equation}%
\begin{align*}
&  \lambda_{2}(\psi,f_{B_{1}})-2\lambda_{3}(\psi,f_{B_{2}})\\
&  =\lambda_{2}\left(  -\varkappa\psi_{2}+\frac{1}{2}(1+\varkappa)\psi
_{1}-\frac{1}{2}(1-\varkappa)\psi_{0}\right)  ,
\end{align*}
and, for $n\geq2,$%
\begin{align}
&  \lambda_{n+1}(\psi,f_{B_{n}})-2\lambda_{n+2}(\psi,f_{B_{n+1}}%
)\label{lambda-kappa}\\
&  =\lambda_{n+1}\left(  -\varkappa\psi_{n+1}+\frac{1}{2}(1+\varkappa)\psi
_{n}-\frac{1}{2^{2}}(1-\varkappa)\psi_{n-1}-...\right. \nonumber\\
&  \left.  -\frac{1}{2^{n}}(1-\varkappa)(\psi_{1}+\psi_{0})\right)  .\nonumber
\end{align}
Thus, applying equations (\ref{sigma-psi eq}) and (\ref{lambda-kappa}) we
obtain
\begin{align*}
&  \frac{-\varkappa\lambda_{n+1}+(\lambda-\sigma_{n+1})}{\lambda_{n+1}}%
\psi_{n+1}+\frac{-2(\lambda-\sigma_{n})+(1+\varkappa)\lambda_{n+1}}%
{2\lambda_{n+1}}\psi_{n}\\
&  =\frac{(1-\varkappa)}{2^{2}}\psi_{n-1}+\left(  \frac{(1-\varkappa)}{2^{3}%
}\psi_{n-2}+...+\frac{(1-\varkappa)}{2^{n}}(\psi_{1}+\psi_{0})\right)
\end{align*}
and%
\begin{align*}
&  \frac{-\varkappa\lambda_{n}+(\lambda-\sigma_{n})}{\lambda_{n}}\psi
_{n}+\frac{-2(\lambda-\sigma_{n-1})+(1+\varkappa)\lambda_{n}}{2\lambda_{n}%
}\psi_{n-1}\\
&  =\frac{(1-\varkappa)}{2^{2}}\psi_{n-2}+\frac{(1-\varkappa)}{2^{3}}%
\psi_{n-3}+...+\frac{(1-\varkappa)}{2^{n-1}}(\psi_{1}+\psi_{0}).
\end{align*}
Let us define variables

\begin{itemize}
\item $A_{n}=-(\lambda-\sigma_{n})+\varkappa^{n+1}$, $n\geq1,$\ 

\item $B_{n}=(1+\frac{\varkappa}{2})(\lambda-\sigma_{n})-(\frac{1}%
{2}+\varkappa)\varkappa^{n+1},$ $n\geq1,$ and

$B_{1}=(\lambda-\sigma_{1})-\frac{\varkappa^{2}}{2}(1+\varkappa)$,
$\ B_{0}=(\lambda-\sigma_{0})-\varkappa^{2}$, $\ $

\item $C_{n}=-\frac{\varkappa}{2}(\lambda-\sigma_{n})+\frac{1}{2}%
\varkappa^{n+2},$ $n\geq1,$and $C_{0}=-\frac{\varkappa^{2}}{2}(1+\varkappa),$
$C_{-1}=0.$
\end{itemize}

The computations from above show that the sequence $\{\psi_{n}\}$ satisfies
the following homogeneous second order difference equation
\[
A_{n+1}\psi_{n+1}+B_{n}\psi_{n}+C_{n-1}\psi_{n-1}=0\text{,\ }%
\]
or equivalently, setting $(\lambda-\sigma_{n})\psi_{n}=\theta_{n}$, we get
\begin{equation}
\theta_{n+1}=D_{n}\theta_{n}+E_{n-1}\theta_{n-1}\text{ \ \ and \ \ }\theta
_{1}=D_{0}\theta_{0}.\text{ } \label{recursion}%
\end{equation}
The coefficients $D_{n}$ and $E_{n-1}$ satisfy
\[
D_{0}=1+O(1)\varkappa^{2},\text{ \ }D_{n}=1+\frac{\varkappa}{2}+O(1)\varkappa
^{n},\text{ \ }E_{n-1}=-\frac{\varkappa}{2}+O(1)\varkappa^{n},
\]
whenever the following condition holds%
\begin{equation}
\lambda\notin\overline{\{\sigma_{n}:n=0,1,...\}}. \label{lambda-condition}%
\end{equation}
If this is the case, by the asymptotic theory of linear second order
difference equations, see \cite[Theorem 1.5]{JanasMoszynski} and \cite[Theorem
8.25 and Corollary 8.27]{Elaydi}, there exist two fundamenthal solutions
$\theta_{1,n}$ and $\theta_{2,n}$ of the equation%
\[
\theta_{n+1}=D_{n}\theta_{n}+E_{n}\theta_{n-1}%
\]
such that asymptotically as $n\rightarrow\infty,$%
\[
\theta_{1,n}=[1+o(1)]\text{ \ and \ }\theta_{1,n}=\left(  \frac{\varkappa}%
{2}\right)  ^{n}[1+o(1)].
\]
Thus, general solution $\theta_{n}$\ of the equation (\ref{recursion})
asymptotically can be written in the form%
\begin{equation}
\theta_{n}=\left(  C_{1}+C_{2}\left(  \frac{\varkappa}{2}\right)  ^{n}\right)
[1+o(1)], \label{general theta}%
\end{equation}
where the constants $C_{1}$ and $C_{2}$ depend on $\theta_{0}$ (remember,
$\theta_{1}=D_{0}\theta_{0}$).

On the other hand, by (\ref{psi-equation}), we get%
\begin{align}
\left\vert \theta_{n}\right\vert  &  \leq\lambda_{n+1}\left\vert
(\psi,f_{B_{n}})\right\vert +\lambda_{n+2}\left\vert (\psi,f_{B_{n+1}%
})\right\vert +...\label{theta ineq}\\
&  \leq\lambda_{n+1}\left\Vert \psi\right\Vert _{L^{\infty}}\left\Vert
f_{B_{n}}\right\Vert _{L^{1}}+\lambda_{n+2}\left\Vert \psi\right\Vert
_{L^{\infty}}\left\Vert f_{B_{n+1}}\right\Vert _{L^{1}}+...\nonumber\\
&  =\left\Vert \psi\right\Vert _{L^{\infty}}\left(  \varkappa^{n+1}%
+\varkappa^{n+2}+...\right)  =O(\varkappa^{n}),\nonumber
\end{align}
in particular, the sequence $\theta_{n}$ tends to zero. Thus, comparing
equations (\ref{general theta}) and (\ref{theta ineq}) we conclude that under
condition (\ref{lambda-condition}),%
\begin{equation}
\theta_{n}=C_{2}\left(  \frac{\varkappa}{2}\right)  ^{n}[1+o(1)],
\label{theta-n}%
\end{equation}
or equivalently,%
\begin{equation}
\psi_{n}=C_{3}\left(  \frac{\varkappa}{2}\right)  ^{n}[1+o(1)], \label{psi-n}%
\end{equation}
for some constant $C_{3}$ which depends on the distance of $\lambda$ to the
set $\overline{\{\sigma_{n}\}}.$

\begin{proposition}
\label{Spectrum-H-plus}$Spec<H|\mathcal{L}_{+}>$ $\subset\overline
{\{\sigma_{n}\}}.$ In particular, if the sequence $\{\sigma_{k}\}$ forms a
dense subset in some interval $\mathcal{I}$,\ then
\[
Spec<H|\mathcal{L}_{+}>=Spec_{ess}<H|\mathcal{L}_{+}>=\mathcal{I}.
\]

\end{proposition}

\begin{proof}
Let us fix $\lambda\notin\mathcal{I}.$ By Proposition \ref{Spec_ess+} it is
enought to show that $\lambda\notin Spec<H|\mathcal{L}_{+}>$. Assume that in
contrary the corresponding generalized $\lambda$-eigenfunction $\psi$ is not
eidentically zero. We already know, see equation (\ref{psi-n}), that $\psi\in
L^{2}(X,m)$. We claim that for any\emph{ }$0<\varkappa<1$ there exists\emph{
}$0<\varepsilon_{\ast}<1$ such that for all\emph{ }$0<\varepsilon
<\varepsilon_{\ast}$ the function\emph{ }$\left\vert \psi\right\vert
^{1-\varepsilon}$ belongs to\emph{ }$L^{1}(X,m)$. Indeed, let us choose
$0<\varepsilon<1$ such that%
\[
\frac{\varepsilon}{1-\varepsilon}<\log_{2}\frac{1}{\varkappa}.
\]
Then, by our choice, $2^{\varepsilon}\varkappa^{1-\varepsilon}<1$. Thus
applying equation (\ref{psi-n}) we get%
\[
\left\Vert \left\vert \psi\right\vert ^{1-\varepsilon}\right\Vert _{L^{1}}\leq
C_{4}\sum2^{n}\left(  \frac{\varkappa}{2}\right)  ^{n(1-\varepsilon)}%
=C_{4}\sum\left(  2^{\varepsilon}\varkappa^{1-\varepsilon}\right)  ^{n}<\infty
\]
as claimed. Next we apply equation (\ref{psi-equation}),%
\begin{align*}
\left\vert \theta_{n}\right\vert  &  \leq\varkappa^{n+1}\left[  \left\vert
(\psi,f_{B_{n}})\right\vert +\varkappa\left\vert (\psi,f_{B_{n}})\right\vert
+...\right] \\
&  \leq\varkappa^{n+1}\left\Vert \left\vert \psi\right\vert ^{1-\varepsilon
}\right\Vert _{L^{1}}\left[  \left\Vert \left\vert \psi\right\vert
^{\varepsilon}f_{B_{n}}\right\Vert _{L^{\infty}}+\varkappa\left\Vert
\left\vert \psi\right\vert ^{\varepsilon}f_{B_{n+1}}\right\Vert _{L^{\infty}%
}+...\right] \\
&  \leq C_{5}\varkappa^{n}\left[  \left(  \frac{\varkappa}{2}\right)
^{n\varepsilon}\frac{1}{2^{n+1}}+\varkappa\left(  \frac{\varkappa}{2}\right)
^{(n+1)\varepsilon}\frac{1}{2^{n+2}}+...\right] \\
&  =C_{6}\left(  \frac{\varkappa}{2}\right)  ^{n}\left(  \frac{\varkappa}%
{2}\right)  ^{n\varepsilon}\left[  1+\frac{\varkappa}{2}+\left(
\frac{\varkappa}{2}\right)  ^{2}+...\right]  \leq C_{7}\left(  \frac
{\varkappa}{2}\right)  ^{(n+1)\varepsilon}.
\end{align*}
This evidently contradicts to (\ref{psi-n}). Thus, $\lambda\notin
Spec<H|\mathcal{L}_{+}>$. The proof is finished.
\end{proof}

\begin{proposition}
\label{Spec_pp+} Assume that $\{\sigma_{n}\}$\ is a bounded sequence. The set
$Spec_{pp}\left\langle H|\mathcal{L}_{+}\right\rangle $, pure point part of
the spectrum of $\left\langle H|\mathcal{L}_{+}\right\rangle $, is an empty
set (Thus, by Propositions \ref{Spec_ess+} and \ref{Spectrum-H-plus},
$Spec\left\langle H|\mathcal{L}_{+}\right\rangle $ coinsides with its singular
continuous part $Spec_{sc}\left\langle H|\mathcal{L}_{+}\right\rangle $).
\end{proposition}

\bigskip

\bigskip


\end{document}